\documentclass[10pt]{amsart}
\usepackage{amsfonts}
\usepackage{amssymb,latexsym}
\usepackage{amsmath}
\usepackage{amsthm}
\usepackage{amssymb}
\usepackage{enumerate}
\newtheorem{theorem}{Theorem}[section]
\newtheorem{lemma}[theorem]{Lemma}
\newtheorem{proposition}[theorem]{Proposition}

\theoremstyle{definition}

\theoremstyle{remark}
\newtheorem{remark}[theorem]{Remark}
\numberwithin{equation}{section}

\begin{document}
\title{The Neumann problem for Parabolic Hessian quotient equations}
\author{Chuanqiang Chen$^1$, Xinan Ma$^2$, Dekai Zhang$^3$}
\address{Chuanqiang Chen, School of Mathematics and Statistics, Ningbo University, Ningbo, 315211, Zhejiang Province, P.R. China.}
\address{Xinan Ma,  School of Mathematical Sciences, University of Science and Technology of China, Hefei, 230026, Anhui Province,  P.R. CHINA}
\address{Dekai Zhang, Shanghai Center for Mathematical Sciences, Jiangwan Campus, Fudan University, No.2005 Songhu Road, Shanghai, P.R. CHINA}
\thanks{ E-mail: $^1$chuanqiangchen@zjut.edu.cn, $^2$ xinan@ustc.edu.cn, $^3$ dkzhang@fudan.edu.cn}
\thanks{$^1$ Research of the first author was supported by ZJNSF NO. LY17A010022 and NSFC NO. 11771396. Research of the second author was supported by  NSFC No. 11721101 and NSFC No. 11871255. Research of the third author was supported by NSFC No.11901102 and  China Postdoctoral Science Foundation No.2019M651333}

\maketitle

\begin{abstract}
In this paper, we consider the Neumann problem for parabolic Hessian quotient equations. We show that the $k$-admissible solution of the parabolic Hessian quotient equation exists for all time and converges to the smooth solution of elliptic Hessian quotient equations. Also the solutions of the classical Neumann problem converge to a translating solution.
\end{abstract}
{\bf Key words}: parabolic Hessian quotient equation, Neumann problem, translating solution.

{\bf 2010 Mathematics Subject Classification}: 35J60, 35B45.

\section{Introduction}

In this paper, we consider the Neumann problem for parabolic Hessian quotient equation, which is of the form
\begin{align} \label{1.1}
\left\{ \begin{array}{l}
u_t= \log\frac{\sigma _k (D^2 u)}{\sigma _l (D^2 u)} - \log{f(x,u)}   \quad \text{in} \quad \Omega \times[0,T),\\
u_\nu = \varphi(x,u)\ \ \ \ \  \qquad \qquad \qquad \text{on} \quad \partial \Omega\times[0,T),\\
u(x,0)=u_0 \qquad \qquad \qquad   \qquad\text{in} \ \ \quad \Omega,
 \end{array}\right.
\end{align}
where $0 \leq l<k \leq n$, $\nu$ is outer unit normal vector of $\partial \Omega$, $T$ is the maximal time, and $\Omega \subset \mathbb{R}^n$, $n\ge2$ is a strictly convex bounded domain with smooth boundary. For any $k = 1, \cdots, n$,
\begin{equation*}
\sigma_k(D^2 u) = \sigma_k(\lambda(D^2 u)) = \sum _{1 \le i_1 < i_2 <\cdots<i_k\leq n}\lambda_{i_1}\lambda_{i_2}\cdots\lambda_{i_k},
\end{equation*}
with $\lambda(D^2 u) =(\lambda_1,\cdots,\lambda_n)$ being the eigenvalues of $D^2 u=: \{  \frac{\partial^2 u}{\partial x_i \partial x_j} \}$. We also set $\sigma_0=1$ for convenience.
And we recall that the G{\aa}rding's cone is defined as
\begin{equation*}
\Gamma_k  = \{ \lambda  \in \mathbb{R}^n :\sigma _i (\lambda ) > 0,\forall 1 \le i \le k\}.
\end{equation*}
For any $C^2$ function $u(x,t)$ (or $u(x)$), if $\lambda(D^2 u) \in \Gamma_k$ holds for any $(x, t) \in \Omega \times (0,T)$ (or $x \in \Omega$), we say $u$ is a $k$-convex function. If the solution $u(x,t)$ of \eqref{1.1} is $k$-convex, then the equation \eqref{1.1} is parabolic and we say $u$ is a $k$-admissible solution of \eqref{1.1}.

If $l=0$, \eqref{1.1} is known as the parabolic $k$-Hessian equation. In particular, \eqref{1.1} is the parabolic Laplace equation if $k =1, l=0$, and the parabolic Monge-Amp\`{e}re equation if $k =n, l=0$. Hessian quotient equation is a more general form of Hessian type equations. It appears naturally in classical geometry, conformal geometry and K\"{a}hler geometry.

 Firstly, we present a brief description for the Dirichlet problem of elliptic equations in $\mathbb{R}^n$. The Dirichlet problem for the Laplace equation is well studied in \cite{ChenWu98, GT}. For nonlinear elliptic equations, the pioneering works have been done by Evans in \cite{E82}, Krylov in \cite{K82, K83, K83a}, Caffarelli-Nirenberg-Spruck in \cite{CNS84, CNS85} and Ivochkina in \cite{I87}. In their papers, they solved the Dirichlet problem for Monge-Amp\`{e}re equations and $k$-Hessian equations elegantly. Since then, many interesting fully nonlinear equations with different structure conditions have been researched, such as Hessian quotient equations, which were solved by Trudinger in \cite{T95}. For more information, we refer the citations of \cite{CNS84}, etc.

For the curvature equations in classical geometry, the existence of hypersurfaces with prescribed Weingarten curvature was studied by Pogorelov \cite{P78}, Caffarelli-Nirenberg-Spruck \cite{CNS85_1, CNS88}, Guan-Guan \cite{GG02}, Guan-Ma \cite{GM03} and the later work by Sheng-Trudinger-Wang \cite{STW04}. The Hessian equation on Riemannian manifolds was also studied by Y.Y. Li \cite{L90}, Urbas \cite{U02} and Guan \cite{G14}. Hessian type equations also appear in conformal geometry, which started from Viaclovsky \cite{Via00},   Chang-Gursky-Yang \cite{CGY02}. In K\"{a}hler geometry, the Hessian equation was studied by Hou-Ma-Wu \cite{HMW10} and Dinew-Kolodziej \cite{DK12}.

Meanwhile, the Neumann and oblique derivative problem of partial differential equations were widely studied. For a priori estimates and the existence theorem of Laplace equation with Neumann boundary condition, we refer to the book \cite{GT}. Also, we recommend the recent book written by Lieberman \cite{L13} for the Neumann and the oblique derivative problems of linear and quasilinear elliptic equations. Especially for the mean curvature equation with prescribed contact angle boundary value problem, Ural'tseva \cite{U75}, Simon-Spruck \cite{SS76} and Gerhardt \cite{G76} got the boundary gradient estimates and the corresponding existence theorem. Recently in \cite{MX16}, the second author and J.J. Xu got the boundary gradient estimates and the corresponding existence theorem for the Neumann boundary value problem on mean curvature equation.

The Yamabe problem with boundary is an important motivation for the study of the Neumann problems. The Yamabe problem on manifolds with boundary was first studied by Escobar, who shows in \cite{E92} that (almost) every compact Riemannian manifold $(M, g)$ is conformally equivalent to one of constant scalar curvature, whose boundary is minimal. The problem reduces to solving the semilinear elliptic critical Sobolev exponent equation with the Neumann boundary condition. It is naturally, the Neumann boundary value problem for Hessian type equations also appears in the fully nonlinear Yamabe problem for manifolds with boundary, which is to find a conformal metric  such that the $k$-th elementary symmetric function of eigenvalues of Schouten tensor is constant and with the constant mean curvature on the boundary of manifold. See Jin-Li-Li \cite{JLL07}, Chen \cite{C07} and Li-Luc \cite{LL14}, but in all these papers they need to impose the manifold are umbilic or total geodesic boundary for $k \geq 2$, which are more like the condition in Trudinger \cite{T87} that the domain is ball.

In 1986, Lions-Trudinger-Urbas solved the Neumann problem of Monge-Amp\`{e}re equations in the celebrated paper \cite{LTU86}. For related results on the Neumann or oblique derivative problem for some class of fully nonlinear elliptic equations can be found in Urbas \cite{U95}. Recently, the second author and G.H. Qiu \cite{MQ15} solved the the Neumann problem of $k$-Hessian equations, and then Chen-Zhang \cite{CZ16} generalized the above result to the the Neumann problem of Hessian quotient equations. Meanwhile, Jiang-Trudinger \cite{JT15, JT16} studied the general oblique boundary value problems for augmented Hessian equations with some regular condition and concavity condition. Motivated by the optimal transport Caffarelli \cite{Ca96} and Urbas \cite{Urbas97} proved the existence of the Monge-Ampere equation with second boundary value problem, for the general convex cost function this second boundary value problem studied by Ma-Trudinger-Wang \cite{MTW}.

If $k =n, l=0$, \eqref{1.1} is the well known parabolic Monge-Amp\`{e}re equation, which relates to the Gauss curvature flow if $f=f(x,u,Du)$.
 O.C. Schn\"{u}rer-K. Smoczyk proved  the long time existence of this Gauss curvature flow and showed that the flow converges to a solution of the prescribed Gauss curvature equation in \cite{SS03}.

Naturally, we want to know how about the Neumann problem of parabolic Hessian quotient equations. In this paper, we obtain two results. One is the long time existence and convergence of solutions of the Neumann problem of parabolic Hessian quotient equation. The other is that the solutions of the classical Neumann problem of parabolic Hessian quotient equation converge to the translating solution.

To state our main results, we first introduce the structural conditions on $\varphi$, $f$ and $u_0$. Firstly, we assume
 \begin{align}\label{1.2}
 \varphi_u:=\frac{\partial\varphi}{\partial u}\le c_{\varphi}<0.
 \end{align}
 and
 \begin{align}\label{1.3}
 f>0 \quad\text{and} \quad f_u\ge0.
 \end{align}
 These two conditions are similar as the Monge-Amp\`{e}re case in \cite{SS04}. Here $u_0$ is always a smooth, $k$-convex function. Moreover, we will always assume either
 \begin{align}\label{1.4}
 \frac{f_u}{f}\ge c_{f}>0,
 \end{align}
 or
 \begin{align}\label{1.5}
\frac{\sigma _k (D^2 u_0)}{\sigma _l (D^2 u_0)}\ge f(x,u_0).
 \end{align}
 We also assume the following compatibility conditions
 \begin{align}\label{1.6}
 (\frac{\partial}{\partial t})^j|_{t=0}(u_{\nu}-\varphi(x,u))=0,\quad \text{for any }\quad j\ge0, \quad\text{on}\quad\partial\Omega.
 \end{align}

 Our first main theorem is
\begin{theorem} \label{th1.1}
Assume that $\Omega $ is a strictly convex bounded  domain in $\mathbb{R}^n$, $n \ge 2$, with smooth boundary.  Let $f, \varphi:\overline\Omega\times\mathbb{R}\rightarrow \mathbb{R}$, be smooth functions which satisfy  $\eqref{1.2}$ and $\eqref{1.3}$. Suppose there is  a smooth, $k$-convex function $u_0$ satisfying the compatibility conditions \eqref{1.6}. We further assume that either $\eqref{1.4}$ or $\eqref{1.5}$ holds. Then there exists a smooth solution $u(x,t)$ of equation \eqref{1.1} for all $t\ge 0$. Moreover, $u(x,t)$ converges smoothly to a smooth function $u^\infty$ which is a solution of  the Neumann  problem for Hessian quotient equation
\begin{align} \label{1.7}
\left\{ \begin{array}{l}
\frac{\sigma _k (D^2 u^\infty)}{\sigma _l (D^2 u^\infty)}  = f(x,u^\infty),  \quad \text{in} \quad \Omega \subset \mathbb{R}^n,\\
u^\infty_\nu(x) = \varphi(x,u^\infty),\qquad \text{on} \quad \partial \Omega,
\end{array} \right.
\end{align}
where $\nu$ is outer unit normal vector of $\partial \Omega$. The rate of convergence is exponential provided $\eqref{1.4}$ holds.
\end{theorem}

Next we consider the related translating solution of the classical Neumann problem for parabolic Hessian quotient equations. The Monge-Amp\`{e}re equation case was proven by \cite{SS04}, and the mean curvature equation by \cite{MWW18}.

Let $u_0$ be a smooth $k$-convex function.  Assume that $u_0\in C^{\infty}(\overline\Omega)$ and satisfies
\begin{align} \label{1.8}
\frac{\partial u_0(x)}{\partial\nu}=\varphi(x) \quad\text{on} \quad\partial\Omega.
\end{align}

\begin{theorem}\label{th1.2}
Let $\Omega$ is a strictly convex bounded domain in $\mathbb{R}^n$ with smooth boundary. Assume that $u_0$ and $\varphi$ are smooth functions satisfying \eqref{1.8}, and $f$ is a positive smooth function, $f\in C^{\infty}(\overline\Omega)$.
Then there exists a smooth $k$-admissible solution $u(x,t)$ of the following equation for all $t\ge 0$.
\begin{align} \label{1.9}
\left\{ \begin{array}{l}
u_t=\log\frac{\sigma _k \left( {D^2 u} \right)}{\sigma _l \left( {D^2 u} \right)} -\log f\left( x \right),  \quad (x, t) \in \Omega \times (0, T), \\
u_{\nu}(x,t) = \varphi (x), \qquad \qquad \qquad\quad x \in \partial \Omega,~ t \in [0, T), \\
u(x,0)=u_0(x), \qquad\qquad \qquad \quad x \in  \Omega,
\end{array} \right.
\end{align}
where $u(\cdot,t)$ approaches $u_0$ in $C^2(\overline\Omega)$ as $t\rightarrow 0$. Moreover, $u(\cdot,t)$ converges smoothly to a translating solution, i.e. to a solution with constant time derivative.
\end{theorem}

The rest of the paper is organized as follows. In Section 2, we collect some  properties and inequalities
of elementary symmetric functions. And we prove the uniform estimate for $|u_t|$ in Section 3. Then we use the uniform estimate of $|u_t|$ to obtain $C^0$-estimate of $u$ in Section 4. The $C^1$-estimate and the $C^2$ estimate are derived in Section 5 and Section 6, respectively. And then we prove Theorem \ref{th1.1} in Section 7. At last, we prove Theorem \ref{th1.2} in Section 8.

\section{preliminary}

In this section, we collect some  properties and inequalities of elementary symmetric functions.

\subsection{Basic properties of elementary symmetric functions}

Let $\lambda=(\lambda_1,\dots,\lambda_n)\in\mathbb{R}^n$ and $1\le k
\le n$.
\begin{equation*}
\sigma_k(\lambda) = \sum _{1 \le i_1 < i_2 <\cdots<i_k\leq n}\lambda_{i_1}\lambda_{i_2}\cdots\lambda_{i_k},
\end{equation*}
We denote by $\sigma _k (\lambda \left| i \right.)$ the symmetric
function with $\lambda_i = 0$ and $\sigma _k (\lambda \left| ij
\right.)$ the symmetric function with $\lambda_i =\lambda_j = 0$.
It is easy to know the following equalities hold
\begin{align*}
&\sigma_k(\lambda)=\sigma_k(\lambda|i)+\lambda_i\sigma_{k-1}(\lambda|i), \quad \forall \,1\leq i\leq n,\\
&\sum_i \lambda_i\sigma_{k-1}(\lambda|i)=k\sigma_{k}(\lambda),\\
&\sum_i\sigma_{k}(\lambda|i)=(n-k)\sigma_{k}(\lambda).
\end{align*}
We also denote by $\sigma _k (W \left|i \right.)$ the symmetric function with $W$ deleting the $i$-row and
$i$-column and $\sigma _k (W \left| ij \right.)$ the symmetric
function with $W$ deleting the $i,j$-rows and $i,j$-columns. Then
we have the following identities.
\begin{proposition}\label{prop2.1}
Suppose $W=(W_{ij})$ is diagonal, and $m$ is a positive integer, then
\begin{align*}
\frac{{\partial \sigma _m (W)}} {{\partial W_{ij} }} = \begin{cases}
\sigma _{m - 1} (W\left| i \right.), &\text{if } i = j, \\
0, &\text{if } i \ne j.
\end{cases}
\end{align*}
\end{proposition}

Recall that the G{\aa}rding's cone is defined as
\begin{equation}\label{2.1}
\Gamma_k  = \{ \lambda  \in \mathbb{R}^n :\sigma _i (\lambda ) >
0,\forall 1 \le i \le k\}.
\end{equation}

\begin{proposition}\label{prop2.2}
Let $\lambda \in \Gamma_k$ and $k \in \{1,2, \cdots, n\}$. Suppose that
$$
\lambda_1 \geq \cdots \geq \lambda_k \geq \cdots \geq \lambda_n,
$$
then we have
\begin{align}
\label{2.2}& \sigma_{k-1} (\lambda|n) \geq \sigma_{k-1} (\lambda|n-1) \geq \cdots \geq \sigma_{k-1} (\lambda|k) \geq \cdots \geq \sigma_{k-1} (\lambda|1) >0; \\
\label{2.3}& \lambda_1 \geq \cdots \geq \lambda_k >  0, \quad \sigma _k (\lambda)\leq C_n^k  \lambda_1 \cdots \lambda_k; \\
\label{2.4}& \lambda _1 \sigma _{k - 1} (\lambda |1) \geq \frac{k} {{n}}\sigma _k (\lambda).
\end{align}
where $C_n^k = \frac{n!}{k! (n-k)!}$.
\end{proposition}

\begin{proof}
All the properties are well known. For example, see \cite{L96} or \cite{HS99} for a proof of \eqref{2.2},
\cite{L91} for \eqref{2.3}, and \cite{CW01} or \cite{HMW10} for \eqref{2.4}.
\end{proof}

\begin{proposition}\label{prop2.3}
(Newton-MacLaurin inequality)
For $\lambda \in \Gamma_k$ and $k > l \geq 0$, $ r > s \geq 0$, $k \geq r$, $l \geq s$, we have
\begin{align} \label{2.5}
\Bigg[\frac{{\sigma _k (\lambda )}/{C_n^k }}{{\sigma _l (\lambda )}/{C_n^l }}\Bigg]^{\frac{1}{k-l}}
\le \Bigg[\frac{{\sigma _r (\lambda )}/{C_n^r }}{{\sigma _s (\lambda )}/{C_n^s }}\Bigg]^{\frac{1}{r-s}},
\end{align}
where $C_n^k = \frac{n!}{k! (n-k)!}$.
\end{proposition}
\begin{proof}
See \cite{S05}.
\end{proof}

\subsection{Key Lemmas}

The following inequalities of Hessian  operators are very useful for us to establish a priori estimates. One can find the proofs in \cite{C15, CZ16}.

\begin{lemma} \label{lem2.4}
Suppose $\lambda = (\lambda_1, \lambda_2, \cdots, \lambda_n) \in \Gamma_k$, $k \geq 1$, and $ \lambda_1 <0$. Then we have
\begin{align} \label{2.6}
\sigma_{m} (\lambda |1)  \geq   \sigma_{m} (\lambda), \quad \forall \quad m =0, 1, \cdots, k.
\end{align}
Moreover, we have
\begin{align} \label{2.7}
\frac{{\partial [\frac{{\sigma _k (\lambda )}}{{\sigma _l (\lambda )}}]}}
{{\partial \lambda _1 }} \geqslant \frac{n}{k}\frac{{k - l}}{{n - l}}\frac{1}
{{n - k + 1}}\sum\limits_{i = 1}^n {\frac{{\partial [\frac{{\sigma _k (\lambda )}}
{{\sigma _l (\lambda )}}]}}{{\partial \lambda _i }}}, \quad \forall \quad 0 \leq l < k.
\end{align}
\end{lemma}

\begin{lemma} \label{lem2.5}
Suppose $A=\{ a_{ij}\}_{n \times n}$ satisfies
 \begin{align} \label{2.8}
a_{11} <0, \quad  \{ a_{ij}\}_{2 \leq i,j \leq n} \quad \text{ is diagonal},
 \end{align}
and $\lambda (A) \in \Gamma_k$ ($k \geq 1$). Then we have
\begin{align} \label{2.9}
\frac{{\partial [\frac{{\sigma _k (A )}}{{\sigma _l (A)}}]}}
{{\partial a_{11} }} \geqslant \frac{n}{k}\frac{{k - l}}{{n - l}}\frac{1}
{{n - k + 1}}\sum\limits_{i = 1}^n {\frac{{\partial [\frac{{\sigma _k (A )}}
{{\sigma _l (A)}}]}}{{\partial a_{ii} }}}, \quad \forall \quad 0 \leq l < k,
\end{align}
and
\begin{align} \label{2.10}
\sum\limits_{i = 1}^n {\frac{{\partial [\frac{{\sigma _k (A )}}
{{\sigma _l (A)}}]}}{{\partial a_{ii} }}}  \geq& \frac{{k - l}}{k} \frac{1}{C_n^l} (-a_{11})^{k-l-1}, \quad \forall \quad 0 \leq l < k,
\end{align}
where $C_n^l = \frac{n!}{l! (n-l)!}$.
\end{lemma}

\begin{lemma} \label{lem2.6}
Suppose $\lambda = (\lambda_1, \lambda_2, \cdots, \lambda_n) \in \Gamma_k$, $k \geq 2$, and $ \lambda_2 \geq \cdots \geq \lambda_n$. If $\lambda_1 > 0$, $\lambda_n < 0$, $\lambda_1 \geq \delta \lambda_2$, and $- \lambda_n \geq \varepsilon \lambda_1$ for small positive constants $\delta$ and $\varepsilon$,  then we have
\begin{align} \label{2.11}
\sigma_{m} (\lambda |1)  \geq  c_0 \sigma_{m} (\lambda), \quad \forall \quad m =0, 1, \cdots, k-1,
\end{align}
where $c_0 = \min \{ \frac{{\varepsilon ^2 \delta ^2}}{{2(n - 2)(n - 1)}},\frac{{\varepsilon ^2 \delta }}{{4(n- 1)}}\}$. Moreover, we have
\begin{align}\label{2.12}
\frac{{\partial [\frac{{\sigma _k (\lambda )}}{{\sigma _l (\lambda )}}]}}
{{\partial \lambda _1 }} \geqslant c_1\sum\limits_{i = 1}^n {\frac{{\partial [\frac{{\sigma _k (\lambda )}}
{{\sigma _l (\lambda )}}]}}{{\partial \lambda _i }}}, \quad \forall \quad 0 \leq l < k,
\end{align}
where $c_1 = \frac{n} {k}\frac{{k - l}}{{n - l}}\frac{c_0^2} {{n - k + 1}}$.
\end{lemma}

\begin{remark}
These lemmas play an important role in the establishment of a priori estimates. Precisely, Lemma \ref{lem2.5} is the key of the gradient estimates in Section 5, including the interior gradient estimate and the near boundary gradient estimate. Lemmas \ref{lem2.4} and Lemma \ref{lem2.6} are the keys of the lower and upper estimates of double normal second order derivatives on the boundary in Section 6, respectively.
\end{remark}

\section{$u_t$-estimate}

In this section, we follow the proof in Schn\"{u}rer-Smoczyk \cite{SS03} to obtain $u_t$-estimate.
\begin{lemma} \label{lem3.1}
Suppose $\Omega \subset \mathbb{R}^n$ is a $C^3$ domain, and $u \in C^{3}(\overline \Omega \times [0, T))$ is a $k$-admissible solution of equation \eqref{1.1}, satisfying \eqref{1.2} and \eqref{1.3}. Then it holds
\begin{align} \label{3.1}
\min \{\min\limits_{ \bar\Omega} u_t \left( {x,0} \right),0\}\le u_t  \left( {x,t} \right) \le \max \{\max\limits_{ \bar\Omega} u_t \left( {x,0} \right),0\}, \quad \forall~ (x,t)\in \Omega\times[0, T).
\end{align}
Moreover,

(i) if $\eqref{1.4}$ holds, then we have for any $0<\lambda<c_f$
\begin{align}\label{3.2}
\min \{\min\limits_{ \bar\Omega} u_t \left( {x,0} \right),0\}\le u_t \left( {x,t} \right)e^{\lambda t} \le \max \{\max\limits_{ \bar\Omega} u_t \left( {x,0} \right),0\},  \quad \forall~ (x,t) \in \Omega\times [0, T).
\end{align}

(ii) if $\eqref{1.5}$ holds, then we have $u_t(x,0) \equiv 0$ or $u_t(x,t) >0$ for any $t>0$.
\end{lemma}

\begin{proof}
For any $\varepsilon >0$ sufficiently small, we consider the evolution equation of $u_t$ in $\Omega\times[0, T-\varepsilon]$. It is easy to see that $u_t$ satisfies
\begin{align*}
(u_t)_t=F^{ij}(u_t)_{ij}-\frac{f_u}{f}u_t,
\end{align*}
where $F^{ij}=\frac{\partial \log{\big(\frac{\sigma_k}{\sigma_l}\big)}}{\partial u_{ij}}$. Assume $u_t(x_0,t_0)=\max\limits_{\bar\Omega\times[0, T-\varepsilon]}{u_t(x,t)}>0$, then the weak parabolic maximum principle implies that either $x_0 \in \partial \Omega$ or $t_0=0$. If $x_0\in\partial\Omega$, we have at $(x_0,t_0)$
\begin{align*}
 0 < (u_t)_{\nu}=\varphi_u u_t\le -c_{\varphi}u_t,
\end{align*}
which is a contradiction. Thus we have $t_0=0$, and the second inequality in \eqref{3.1} is proved. Similarly, we can prove the first inequality in \eqref{3.1}.

The proof of \eqref{3.2} is similar as that in \cite{SS03}. We produce it here for completeness. Let $v(x,t)=e^{\lambda t}u_t(x,t)$ for $0<\lambda<c_f$, and then $v(x,t)$ satisfies
\begin{align*}
v_t&=\lambda v+e^{\lambda t}u_{tt}\\
&=\lambda v+e^{\lambda t}(F^{ij}u_{tij}-\frac{f_u}{f}u_t)\\
&=F^{ij}v_{ij}+(\lambda-\frac{f_u}{f})v.
\end{align*}
Assume $v(x_0,t_0)=\max\limits_{\bar\Omega\times[0, T-\varepsilon]}{v(x,t)}>0$. Therefore by maximum principle, we have either $x_0 \in \partial \Omega$ or $t_0=0$. If $x_0\in \partial\Omega$, we have at $(x_0,t_0)$ by Neumman boundary condition,
\begin{align*}
0 < v_{\nu}=e^{\lambda t}u_{\nu t}=v\varphi_u<0,
\end{align*}
which is a contradiction. Hence, $t_0=0$, i.e. $\max\limits_{\bar\Omega\times[0, T-\varepsilon])}{v(x,t)}=\max\limits_{\bar\Omega}{u_t(x,0)}$. Similarly, we can prove the first inequality in \eqref{3.2}.

At last, if $\eqref{1.5}$ holds, then we have $u_t(x,0) \geq 0$. From \eqref{3.1}, we know that $u_t \geq 0$ for any $t>0$. If $u_t(x,0)$ is not identically to zero, then we let $u_t(x_0,t_0)=\min\limits_{\bar\Omega\times[\varepsilon, T-\varepsilon]} u_t(x,t)$.  If $u_t(x_0,t_0)=0$, then the strong maximum principle tells us that $u_t(x,t)\equiv u_t(x_0,t_0)=0$, for any $(x,t)\in \bar\Omega \times [0,t_0)$. Thus $u_t(x,0) \equiv 0$, which  contradicts the hypothesis that $u_t(x,0)$ is not identically to zero.
\end{proof}

\section{$C^0$ estimate}

Due to $u_t$-estimate in Lemma \ref{lem3.1}, we can derive the $C^0$-estimate of $u$ as follows.
\begin{theorem} \label{th4.1}
Suppose $\Omega \subset \mathbb{R}^n$ is a $C^3$ domain, and $u \in C^{3}(\overline \Omega \times [0, T))$ is a $k$-admissible solution of equation \eqref{1.1}, satisfying \eqref{1.2} and \eqref{1.3}. Moreover, if $f$ satisfies \eqref{1.4} or $u_0$ satisfies \eqref{1.5}, then we have
\begin{align}\label{4.1}
|u(x,t)|\le M_0, \quad (x,t) \in \Omega \times [0, T),
\end{align}
where $M_0$ is a positive constant depending only on $c_\varphi$, $\max\limits_{\bar\Omega}|\varphi(x,0)| $, $\max\limits_{\bar\Omega}|u_0|$, $c_f$ and $\max\limits_{\bar\Omega}|u_t(x,0)|$.
\end{theorem}
\begin{proof}
We first prove the upper bound of $u$. For any fixed $t$, $u(x,t)$ is a subharmonic function. If $u(x_0,t)=\max\limits_{\overline\Omega}u(x,t)>0$, we must have $x_0\in\partial\Omega$. By the Neumann boundary condition, we have at this point
\begin{align*}
0\le u_{\nu}=\varphi(x_0,u)=\varphi(x_0,0)+\varphi_u(x_0,\theta)u \le \varphi(x_0,0)-c_\varphi u.
\end{align*}
Thus $u(x_0,t)\le\frac{\max\limits_{\bar\Omega}\varphi(x,0)}{c_\varphi}$.

Next we prove the lower bound of $u$. If \eqref{1.5} holds, by the equation $u_t(x, 0)=\log{\big(\frac{\sigma_k(D^2u_0)}{\sigma_l(D^2u_0)} \big)}-\log{f(x, u_0)}\ge0$. Thus by Lemma \ref{lem3.1},  we immediately have
$u(x,t)\ge u(x,0)=u_0(x)$.

If \eqref{1.4} holds, we have by Lemma \ref{lem3.1}
\begin{align*}
u(x,t)&\ge u(x,0)+\int^t_0{u_t(x,s)}ds\\
&\ge u_0(x)+\min\{\min\limits_{\bar\Omega}u_t(x,0),0\}\int^t_0{e^{-\frac{c_f}{2} s}}ds\\
&\ge u_0(x)+\frac{2\min\{\min\limits_{\bar\Omega}u_t(x,0),0\}}{c_f}.
\end{align*}
Hence \eqref{4.1} holds if we choose $M_0 = \frac{ \max\limits_{\bar\Omega}\varphi(x,0) }{ c_\varphi} + \max\limits_{\bar\Omega}|u_0(x)| +\frac{2\max\limits_{\bar\Omega}|u_t(x,0)|}{c_f}$.
\end{proof}

\begin{remark}
Due to the $C^0$ estimate of $u$ and $u_t$, we now have
\begin{align} \label{4.2}
\frac{\sigma_k(D^2u)}{\sigma_l(D^2 u)} = f(x,u)e^{u_t} \ge \min\limits_{\bar\Omega}f(x,-M_0) \cdot e^{-\max\limits_{\bar\Omega}|u_t(x,0)|}=:c_2>0.
\end{align}
\end{remark}

\section{$C^1$ estimates}

In this section, we prove the global gradient estimate as follows
\begin{theorem} \label{th5.1}
Suppose $\Omega \subset \mathbb{R}^n$ is a $C^3$ domain, and $u \in C^{3}(\overline \Omega \times [0, T))$ is a $k$-admissible solution of equation \eqref{1.1}, satisfying \eqref{1.2} and \eqref{1.3}. Then we have
\begin{align}\label{5.1}
\sup\limits_{\Omega \times [0, T)} |D u|  \leq M_1,
\end{align}
where $M_1$ depends on $n$, $k$, $l$, $\Omega$, $ |u|_{C^{0}}$, $ |u_t|_{C^{0}}$, $ |Du_0|_{C^{0}}$, $|\varphi|_{C^{3}}$, $\min f$, $\max f$ and $|D_x f|_{C^{0}}$.
\end{theorem}

To state our theorems, we denote $d(x) = \textrm{dist}(x, \partial \Omega)$, and $\Omega_{\mu} = \{ x \in \Omega| d(x) < \mu \}$ where $\mu$ is a small positive universal constant to be determined in Theorem \ref{th5.3}. In Subsection 5.1, we give the interior gradient estimate in $(\Omega \setminus \Omega_{\mu}) \times [0, T)$, and in Subsection 5.2 we establish the near boundary gradient estimate in $\Omega_{\mu} \times [0, T)$, following the idea of  Ma-Qiu \cite{MQ15}.

\subsection{Interior gradient estimate}

\begin{theorem} \label{th5.2}
Under the assumptions in Theorem \ref{th5.1}, then we have
\begin{align}\label{5.2}
\sup_{(\Omega \setminus \Omega_{\mu}) \times [0, T)} |D u|  \leq \widetilde{M_1},
\end{align}
where $\widetilde{M_1}$ depends on $n$, $k$, $l$, $\mu$, $M_0$, $ |D u_0|_{C^{0}}$, $|u_t|_{C^{0}}$, $\min f$, $\max f$ and $|D_x f|_{C^{0}}$.
\end{theorem}

\begin{proof}
For any fixed point $(x_0, t_0) \in (\Omega \setminus \Omega_{\mu}) \times (0, T)$, we prove that $ |D u|(x_0, t_0)  \leq {\widetilde{M_1}}$, where $\widetilde{M_1}>0$ depends on $n$, $k$, $l$, $\mu$, $M_0$, $ |D u_0|_{C^{0}}$, $|u_t|_{C^{0}}$, $\min f$, $\max f$ and $|D_x f|_{C^{0}}$. It is easy to know that $B_{\mu} (x_0) \times [0, t_0] \subset \Omega \times [0, t_0]$, and we consider the auxiliary function
\begin{align} \label{5.3}
G(x,t) = |D u|\psi (u)\rho (x)
\end{align}
in $B_{\mu} (x_0) \times [0, t_0]$, where $\rho  = \mu^2 - |x - x_0|^2 $, and $\psi (u) = (3 M_0- u)^{ - \frac{1}{2}}$.
Then $G(x,t)$ attains maximum at some point $(x_1, t_1)\in B_{\mu} (x_0) \times [0, t_0]$.
If $t_1 =0$, then $|D u(x_1, t_1)| = |D u_0(x_1)|$. It is easy to obtain the estimate \eqref{5.2}.

In the following, we assume $t_1 \in (0, t_0]$. By rotating the coordinate $(x_1, \cdots, x_{n})$, we can assume
\begin{align}\label{5.4}
u_1(x_1, t_1) = |D u|(x_1, t_1) >0,  \quad \{ u_{ij} \}_{2 \leq i, j \leq n} \text{ is diagonal}.
\end{align}
Then
\begin{align} \label{5.5}
 \phi (x,t)  = \log u_1 (x,t) + \log \psi (u) + \log \rho
\end{align}
attains local maximum at the point $(x_1, t_1)\in B_{\mu} (x_0) \times [0, t_0]$.
Denote $\widetilde \lambda = (\widetilde \lambda_2, \cdots, \widetilde\lambda_{n})=(u_{22}(x_1, t_1), \cdots, u_{nn}(x_1, t_1))$, and all the calculations are at $(x_1, t_1)$. So we have at $(x_1, t_1)$,
\begin{align}
 \label{5.6}0 =& \phi _i  = \frac{{u_{1i} }}{{u_1 }} + \frac{{\psi _i }}{\psi } + \frac{{\rho _i }}{\rho }, \\
 \label{5.7}0 \leq& \phi _t  = \frac{{u_{1t} }}{{u_1 }} + \frac{{\psi _t}}{\psi }.
 \end{align}
Hence
\begin{align}\label{5.8}
 \frac{{u_{11} }}{{u_1 }} =  - (\frac{{\psi _1 }}{\psi } + \frac{{\rho _1 }}{\rho }) =  - (\frac{{\psi '}}{\psi }u_1  + \frac{{\rho _1 }}{\rho }).
 \end{align}

In the following, we always assume
\begin{align}\label{5.9}
 |D u (x_0, t_0)| \ge 32 \sqrt{2}\frac{{M_0}}{\mu}.
\end{align}
Otherwise there is nothing to prove. Then we have
\begin{align}\label{5.10}
u_1 (x_1, t_1 )\rho (x_1 ) \ge& \frac{{\psi (u)(x_0, t_0)}}{{\psi (u(x_1, t_1))}}|D u|(x_0, t_0)\rho (x_0)
\ge \sqrt{\frac{1}{2}} \cdot 32 \sqrt{2} \frac{{M_0}}{\mu} \cdot \mu^2  = 2 \cdot 8 M_0 \cdot 2\mu \notag\\
\ge& 2\frac{\psi }{{\psi '}}|\rho _1 |,
\end{align}
so
\begin{align} \label{5.11}
 \frac{{u_{11} }}{{u_1 }} =  - (\frac{{\psi '}}{\psi }u_1  + \frac{{\rho _1 }}{\rho }) =  - \frac{{\psi '}}{{2\psi }}u_1  - \frac{{\psi 'u_1 \rho  + 2\psi \rho _1 }}{{2\psi \rho }} \le  - \frac{{\psi '}}{{2\psi }}u_1.
\end{align}

Denote $F^{ij} = \frac{\partial \log \big(\frac{\sigma _k (D^2 u)}{\sigma _l (D^2 u)}\big)} {{\partial u_{ij} }}$, and we have
\begin{align}\label{5.12}
0 \ge F^{ij} \phi _{ij} -\phi _{t}  =& F^{ij} [\frac{{u_{1ij} }}{{u_1 }} - \frac{{u_{1i} u_{1j} }}{{u_1 ^2 }} + \frac{{\psi _{ij} }}{\psi } - \frac{{\psi _i \psi _j }}{{\psi ^2 }} + \frac{{\rho _{ij} }}{\rho } - \frac{{\rho _i \rho _j }}{{\rho ^2 }}] - [\frac{{u_{1t} }}{{u_1 }} + \frac{{\psi _t}}{\psi }]\notag \\
=& \frac{1}{{u_1 }}\frac{f_{1 }+ f_u u_1}{f} - F^{ij} \frac{{u_{1i} u_{1j} }}{{u_1 ^2 }} + [\frac{{\psi ''}}{\psi } - \frac{{\psi '^2 }}{{\psi ^2 }}]F^{ij} u_i u_j  + \frac{{\psi '}}{\psi }[(k-l)-u_t] \notag \\
 &+ F^{ij} \frac{{\rho _{ij} }}{\rho } - F^{ij} \frac{{\rho _i \rho _j }}{{\rho ^2 }} \notag\\
=& \frac{1}{{u_1 }}\frac{f_{1 }+ f_u u_1}{f} + [\frac{{\psi ''}}{\psi } - 2\frac{{\psi '^2 }}{{\psi ^2 }}]F^{11} u_1 ^2  + \frac{{\psi '}}{\psi }[(k-l)-u_t] \notag\\
&- \sum\limits_i {F^{ii} } \frac{2}{\rho } - F^{ij} [\frac{{\psi _i }}{\psi }\frac{{\rho _j }}{\rho } + \frac{{\rho _i }}{\rho }\frac{{\psi _j }}{\psi }] - 2F^{ij} \frac{{\rho _i \rho _j }}{{\rho ^2 }} \notag\\
\ge& -\frac{|f_1/f|}{{u_1 }} + [\frac{{\psi ''}}{\psi } - 2\frac{{\psi '^2 }}{{\psi ^2 }}]F^{11} u_1 ^2  + \frac{{\psi '}}{\psi }[(k-l)-u_t] \notag\\
&- \sum\limits_i {F^{ii} } \frac{2}{\rho } -2 \sum\limits_i {F^{ii} } \frac{{\psi '  }}{\psi }u_1\frac{{ |D \rho|}}{\rho } - 2 \sum\limits_i {F^{ii} } \frac{{ | D \rho |^2}}{{\rho ^2 }} \notag\\
\ge& -\frac{|f_1/f|}{{u_1 }} + \frac{1}{64 M_0^{2}} F^{11} u_1 ^2  + \frac{{\psi '}}{\psi }[(k-l)-u_t]  \notag \\
 & -  2\sum\limits_i {F^{ii} } \Big[ \frac{1}{\rho }+ \frac{1}{ 4 M_0 }u_1\frac{{2 \mu}}{\rho } +  \frac{{4 \mu^2 }}{{\rho ^2 }}\Big].
\end{align}
From Lemma \ref{lem2.5}, we know
\begin{align}\label{5.13}
F^{11} \geq c_3 \sum {F^{ii} },
\end{align}
where $c_3 =\frac{{n(k - l)}}{{k(n - l)}}\frac{1}{{n - k + 1}}$. Moreover, from \eqref{2.10} and \eqref{5.11}, we have
\begin{align}\label{5.14}
\sum {F^{ii} } = \sum {\frac{1}{\frac{\sigma _k (D^2 u)}{\sigma _l (D^2 u)}} \frac{\partial [ \frac{\sigma _k (D^2 u)}{\sigma _l (D^2 u)}]} {{\partial u_{ii} }}}=& \frac{1}{f e^{u_t}} \sum { \frac{\partial [ \frac{\sigma _k (D^2 u)}{\sigma _l (D^2 u)}]} {{\partial u_{ii} }}} \notag \\
\geq& \frac{1}{f e^{u_t}} \frac{{k - l}}{k} \frac{1}{C_n^l} (-u_{11})^{k-l-1} \notag \\
\geq& c_4 u_{1}^{2(k-l-1)}.
\end{align}
Then we can get
\begin{align}\label{5.15}
0 \ge F^{ij} \phi _{ij} -\phi _{t}  \ge& \frac{1}{128 M_0^{2}} c_3 c_4 u_{1}^{2(k-l)} - \frac{ |f_{1 }/f|}{{u_1 }} + \frac{{\psi '}}{\psi }[(k-l)-u_t]   \notag \\
 & + \sum\limits_i {F^{ii} } \Big[\frac{1}{128 M_0^{2}}  c_3 u_1 ^2  -2\big(\frac{1}{\rho }+ \frac{1}{ 4 M_0 }u_1\frac{{2 \mu}}{\rho } +  \frac{{4 \mu^2 }}{{\rho ^2 }}\big)\Big].
\end{align}
This yields
\begin{align}\label{5.16}
\rho (x_1 )u_1 (x_1, t_1)  \leq C,
\end{align}
where $C$ depends on $n$, $k$, $l$, $\mu$, $M_0$, $|u_t|_{C^{0}}$, $\min f$, $\max f$ and $|D_x f|_{C^{0}}$. Hence we can get
\begin{align}
|D u (x_0, t_0)| \leq \frac{{\psi (u)(x_1, t_1)}}{{\psi (u)(x_0, t_0)}}  \frac{1}{\rho (x_0)} \rho (x_1 )u_1 (x_1, t_1) \leq \widetilde{M_1}. \notag
\end{align}
From this, the proof is complete.
\end{proof}

\subsection{ Near boundary gradient estimate}

\begin{theorem} \label{th5.3}
Under the assumptions in Theorem \ref{th5.1}, then there exists a positive universal constant $\mu$ such that
\begin{align}\label{5.17}
\sup_{\Omega_{\mu}\times [0, T)} |D u|  \leq \max\{\widetilde{M_1}, \widehat{M_1}\},
\end{align}
where $\widetilde{M_1}$ is the constant in Theorem \ref{th5.2}, and $\widehat{M_1}$ depends on $n$, $k$, $l$, $\mu$, $\Omega$, $M_0$, $ |Du_0|_{C^{0}}$, $|u_t|_{C^{0}}$, $\min\limits_{\Omega} f$, $\max f$, $|D_x f|_{C^{0}}$ and $|\varphi|_{C^3}$.
\end{theorem}

\begin{proof}
The proof follows the idea of Ma-Qiu \cite{MQ15}.

Since $\Omega$ is a $C^3$ domain, it is well known that there exists a small positive universal constant $0<{\mu} < \frac{1}{10}$ such that
$d(x) \in C^3(\overline{\Omega_{\mu}})$. As in Simon-Spruck \cite{SS76} or Lieberman \cite{L13} (in page 331), we can
extend $\nu$ by $\nu = - D d$ in $\Omega_{\mu}$ and note that $\nu$ is a $C^2(\overline{\Omega_{\mu}})$ vector field. As mentioned
in the book \cite{L13}, we also have the following formulas
\begin{align}
\label{5.18}&|D \nu| + |D^2 \nu| \leq C_0, \quad \text{in} \quad \overline{\Omega_{ \mu}}, \\
\label{5.19}&\sum_{i=1}^n \nu^i D_j \nu^i =0, \quad \sum_{i=1}^n \nu^i D_i \nu^j =0, \quad |\nu| =1,  \quad \text{in} \quad \overline{\Omega_{\mu}},
\end{align}
where $C_0$ is depending only on $n$ and $\Omega$. As in \cite{L13}, we define
\begin{align}\label{5.20}
c^{ij} = \delta_{ij} - \nu^i \nu^j,  \quad \text{in} \quad \overline{\Omega_{ \mu}},
\end{align}
and for a vector $\zeta \in \mathbb{R}^n$, we write $\zeta'$ for the vector with $i$-th component $\sum_{j=1}^n c^{ij} \zeta^j$.
Then we have
\begin{align}\label{5.21}
|(D u)'|^2 = \sum_{i,j=1}^n c^{ij} u_i u_j,  \quad \text{and} \quad |D u|^2 =|(D u)'|^2 + u_\nu ^2.
\end{align}

We consider the auxiliary function
\begin{align} \label{5.22}
G(x,t)  = \log |D w|^2  +h(u) + g(d),
\end{align}
where
\begin{align*}
&w(x,t) = u(x,t) + \varphi(x, u) d(x),  \\
&h(u) = -\log(1+M_0 -u),\\
&g(d)= \alpha_0 d,
\end{align*}
with $\alpha_0 > 0$ to be determined later. Note that here $\varphi \in C^3(\overline \Omega)$ is an extension with universal $C^3$ norms.

For any fixed $T_0 \in (0, T)$, it is easy to know $G(x,t)$ is well-defined in $\overline{\Omega_{\mu}} \times [0, T_0]$. Then we assume that $G(x,t)$ attains its maximum at a point $(x_0, t_0) \in \overline{\Omega_{\mu}} \times [0, T_0]$. If $t_0 =0$, then $|D u(x_0, t_0)| = |D u_0(x_0)|$. It is easy to obtain the estimate \eqref{5.17}.

In the following, we assume $t_0 \in (0, T_0]$ and $|D u|(x_0, t_0) > 10n[|\varphi|_{C^1} + |u|_{C^0}]$. Then we have
\begin{align}\label{5.23}
 \frac{1}{2} |D u(x_0, t_0)|^2 \leq |D w(x_0, t_0)|^2 \leq 2 |D u(x_0, t_0)|^2,
 \end{align}
since $w_i = u_i + D_i \varphi d + \varphi d_i$ and $d$ is sufficiently small. Now we divide into three cases to complete the proof of Theorem \ref{th5.3}.

$\blacklozenge$ CASE I: $x_0 \in \partial \Omega_\mu \cap \Omega$.

Then $x_0 \in \Omega \setminus \Omega_\mu$, and  we can get from the interior gradient estimate (i.e. Theorem \ref{th5.2}),
\begin{align}\label{5.24}
|D u| (x_0, t_0) \leq \sup_{(\Omega \setminus \Omega_{\mu})\times [0, T)} |D u|  \leq \widetilde{M_1},
\end{align}
then we can prove \eqref{5.17}.

$\blacklozenge$ CASE II: $x_0 \in \partial \Omega$.

At $(x_0, t_0)$, we have $d=0$. We may assume $x_0=0$. We choose the coordinate such that  $\partial\Omega$ is locally represented by $\partial\Omega=(x',\rho(x'))$, where $\rho(x')$ is a $C^3$ function with $\rho_i(0)=0$ for $1\le i\le n-1$. Thus $\nu(x_0)=(0,\cdots,0,-1)$. Rotating the $x'$-coordinate further, we can assume $w_1(x_0,t_0)=|Dw|(x_0,t_0)$. Therefore we have
\[
w_n(x_0, t_0)  = u_n   +D_n \varphi d + \varphi d_n  = u_n + \varphi  = 0
\]
By Hopf lemma, we have
\begin{align}\label{5.25}
 0 \ge  - G_\nu  (x_0, t_0) = G_n (x_0, t_0) =& \frac{2w_k w_{kn} }{|Dw|^2} + h'u_n  + g' \notag\\
  =& \frac{2w_{1n}}{{w_1} } - h'\varphi  + g'.
\end{align}
Differentiate the Neumann boundary condition along its tangential direction $e_1$ at $(x_0, t_0)$,
\begin{align}\label{5.26}
- u_{n1}  + u_i {\nu ^i} _{,1}  = u_{i1} \nu ^i  + u_i {\nu ^i} _{,1}  = D_1 \varphi.
\end{align}
Since
\begin{align}\label{5.27}
 w_{1n}  =& u_{1n}  + \varphi d_{1n} +  D_1 \varphi d_n + D_n \varphi d_{1}\notag \\
  =& u_{1n}  + \varphi d_{1n} +  D_1 \varphi.
\end{align}
Hence
\begin{align}\label{5.28}
w_{1n}  =&  u_i {\nu ^i} _{,1}  + \varphi d_{1n}  \notag  \\
\geq& - |D\nu| |Du| - |D^2 d| |\varphi|  \notag   \\
\ge&-C_1 w_1,
\end{align}
where $C_1 =: 4 |D\nu|_{C^0} + 2 |D^2 d|_{C^0}$. Inserting \eqref{5.28} into \eqref{5.25}, we have
\begin{align*}
0 \ge -2C_1-\max\limits_{\Omega}|\varphi| + \alpha_0,
\end{align*}
which is a contradiction if we choose $\alpha_0> 2 C_1 + \max\limits_{\Omega}|\varphi|$.

$\blacklozenge$ CASE III: $x_0 \in \Omega_\mu$.

At $(x_0, t_0)$, we have $0 < d(x_0) < \mu$, and by rotating the coordinate $\{e_1, \cdots, e_{n}\}$, we can assume
\begin{align}\label{5.29}
w_1(x_0, t_0) = |D w|(x_0, t_0) >0,  \quad \{ u_{ij}(x_0, t_0) \}_{2 \leq i, j \leq n} \text{ is diagonal}.
\end{align}
In the following, we denote $\widetilde \lambda = (\widetilde \lambda_2, \cdots, \widetilde\lambda_{n})=(u_{22}(x_0, t_0), \cdots, u_{nn}(x_0, t_0))$, and all the calculations are at $(x_0, t_0)$. So from the definition of $w$, we know $w_i  = u_i  + \varphi d_i  + (\varphi _i +\varphi_u  u_i)d$, and we get
\begin{align}
\label{5.30}& u_1  = \frac{{w_1  - \varphi d_1  - \varphi _1 d}}{{1 + \varphi_u d}} >0, \\
\label{5.31}& u_i  = \frac{{ - \varphi d_i  - \varphi _i d }}{{1 +\varphi_u d}}, \quad  i \ge 2.
\end{align}
By the assumption $|D u|(x_0, t_0) > 10n[|\varphi|_{C^0} + |D_x \varphi|_{C^0}]$ and $d$ sufficiently small, we know for $i \ge 2$
\begin{align}\label{5.32}
|u_i|\leq \frac{{ |\varphi | + |\varphi _i|}}{{1 +\varphi_u d}} \leq \frac{1}{9n}|D u|(x_0, t_0) ,
\end{align}
hence
\begin{align}\label{5.33}
u_1  = \sqrt{|D u|^2 - \sum_{i=2}^n u_i^2} \geq \frac{1}{2} |D u| \geq \frac{1}{4} w_1.
\end{align}

Also we have at $(x_0, t_0)$,
\begin{align}
\label{5.34} 0 =& G_i  = \frac{{(|D w|^2 )_i }}{{|D w|^2 }} + h'u_i  + \alpha _0 d_i, \quad i = 1, 2, \cdots, n, \\
\label{5.35} 0 \geq& G_t  = \frac{{(|D w|^2 )_t }}{{|D w|^2 }} + h'u_t.
\end{align}
hence
\begin{align}\label{5.36}
\frac{{2w_{1i} }}{{w_1 }} =  - [h'u_i  + \alpha _0 d_i ], \quad i = 1, 2, \cdots, n.
\end{align}

From the definition of $w$, we know
\begin{align}
\label{5.37} w_{1i}  =& (1 +\varphi_u d)u_{1i}  + \varphi d_{1i}  + (\varphi _{1i}+\varphi _{1u} u_i+\varphi _{iu} u_1+ \varphi _{uu} u_1 u_i) d  \notag \\
  &+ (\varphi _1 + \varphi _{u} u_1 ) d_i  + (\varphi _i  +\varphi _{u} u_i )d_1,   \quad i =1, \cdots, n; \\
\label{5.38} w_{1t}  =& (1 +\varphi_u d)u_{1t}  + (\varphi _{1u} u_t+ \varphi _{uu} u_1 u_t)d +\varphi _{u} u_t d_1.
\end{align}
So we have
\begin{align}\label{5.39}
u_{11}  =& \frac{{w_{11} }}{{1 +\varphi_u d}} - \frac{{(\varphi _{11}+2\varphi _{1u} u_1+ \varphi _{uu} u_1^2) d + \varphi d_{11} + 2(\varphi _1 + \varphi _{u} u_1 ) d_1 }}{{1 +\varphi_u d}}  \notag \\
=& \frac{{ - [h'u_1  + \alpha _0 d_1 ]w_1 }}{{2(1  +\varphi_u d)}}- \frac{{ \varphi _{uu}d}}{{1 +\varphi_u d}} u_1^2 - \frac{{ 2\varphi _{1u} d+  2\varphi _{u} d_1 }}{{1 +\varphi_u d}} u_1 - \frac{{\varphi _{11}d + \varphi d_{11} + 2\varphi _1  d_1 }}{{1 +\varphi_u d}}  \notag \\
\le& \frac{{ - h'}}{{2(1 +\varphi_u d)}}u_1 w_1 + \frac{{ |\varphi _{uu}|d}}{{1 +\varphi_u d}} u_1^2 \notag \\
&+ \frac{{\alpha _0 }}{{2(1 +\varphi_u d)}}w_1 + \frac{{ 2|\varphi _{1u}| d+  2|\varphi _{u} ||d_1| }}{{1 +\varphi_u d}} u_1 + \frac{{|\varphi _{11}|d + \varphi |d_{11}| + 2|\varphi _1 || d_1| }}{{1 +\varphi_u d}}  \notag \\
\le& - \frac{1}{{16(1 + 2M_0 )}}w_1 ^2  < 0,
\end{align}
since $w_1$ is sufficiently large and $d$ is sufficiently small. Moreover, for $i =1, \cdots, n$, we can get
\begin{align}\label{5.40}
|u_{1i} | =& \Big|\frac{{w_{1i} }}{{1 +\varphi_u d}} - \frac{{(\varphi _{1i}+\varphi _{1u} u_i+\varphi _{iu} u_1+ \varphi _{uu} u_1 u_i)  d + \varphi d_{1i} + D_1\varphi d_i + D_i\varphi d_1}}{{1 +\varphi_u d}}\Big| \notag \\
=& \Big|\frac{{ - [h'u_i  + \alpha _0 d_i ]w_1 }}{{2(1  + \varphi _{u} d)}}  - \frac{{(\varphi _{1i}+\varphi _{1u} u_i+\varphi _{iu} u_1+ \varphi _{uu} u_1 u_i) d + \varphi d_{1i} + D_1\varphi d_i + D_i\varphi d_1}}{{1 +\varphi_u d}} \Big|  \notag \\
\le& C_2 w_1 ^2.
\end{align}

Denote $F^{ij}  = \frac{{\partial [ \log \frac{{\sigma _k (D^2 u )}}{{\sigma _l (D^2 u )}}]}}{{\partial u_{ij} }}$. Then we have
\begin{align}
G_{ij}  = \frac{{(|D w|^2 )_{ij} }}{{|D w|^2 }} - \frac{{(|D w|^2 )_i }}{{|D w|^2 }}\frac{{(|D w|^2 )_j }}{{|D w|^2 }} + h'u_{ij}  + h''u_i u_j  + \alpha _0 d_{ij}, \notag
\end{align}
and
\begin{align}\label{5.41}
0 \ge& \sum\limits_{ij = 1}^n {F^{ij} G_{ij} } - G_t = \frac{{F^{ij} (|D w|^2 )_{ij} }}{{|D w|^2 }} - F^{ij} \frac{{(|D w|^2 )_i }}{{|D w|^2 }}\frac{{(|D w|^2 )_j }}{{|D w|^2 }}  - \frac{{(|D w|^2 )_t }}{{|D w|^2 }}\notag \\
&\qquad \qquad \qquad \qquad \quad + F^{ij} [h'u_{ij}  + h''u_i u_j  + \alpha _0 d_{ij} ]-  h'u_t \notag \\
=& \frac{{2F^{ij} [\sum\limits_{p = 2}^n {w_{pi} w_{pj} }  + w_{1i} w_{1j}  + w_1 w_{1ij} ]}}{{w_1 ^2 }} - F^{ij} \frac{{2w_{1i} }}{{w_1 }}\frac{{2w_{1j} }}{{w_1 }} -\frac{{2w_{1t} }}{{w_1 }} \notag \\
&+ F^{ij} [h'u_{ij}  + h''u_i u_j  + \alpha _0 d_{ij} ]-  h'u_t \notag \\
\ge& \frac{{2F^{ij} w_{1ij} }}{{w_1 }} -\frac{{2w_{1t} }}{{w_1 }}- \frac{1}{2}F^{ij} \frac{{2w_{1i} }}{{w_1 }}\frac{{2w_{1j} }}{{w_1 }}+ F^{ij} [h'u_{ij}  + h''u_i u_j  + \alpha _0 d_{ij} ] -  h'u_t\notag \\
=& \frac{{2F^{ij} w_{1ij} }}{{w_1 }} -\frac{{2w_{1t} }}{{w_1 }} - \frac{1}{2}F^{ij} [h'u_i  + \alpha _0 d_i ][h'u_j  + \alpha _0 d_j ] \notag \\
&+ h'[(k-l)-u_t] + F^{ij} [h''u_i u_j  + \alpha _0 d_{ij} ] \notag  \\
\ge& \frac{{2F^{ij} w_{1ij} }}{{w_1 }} -\frac{{2w_{1t} }}{{w_1 }}+ h'[(k-l)-u_t] \notag \\
&+ F^{ij} [(h'' - \frac{1}{2}h'^2 )u_i u_j  - \alpha _0 h'd_i u_j  + \alpha _0 d_{ij}  - \frac{1}{2}\alpha _0 d_i d_j ].
\end{align}
It is easy to know
\begin{align}\label{5.42}
 &F^{ij} [(h'' - \frac{1}{2}h'^2 )u_i u_j  - \alpha _0 h'd_i u_j  + \alpha _0 d_{ij}  - \frac{1}{2}\alpha _0 d_i d_j ]  \notag \\
\ge& \frac{1}{{2(1 + 2M_0 )}}[F^{11} u_1 ^2  -2 \sum\limits_{i=2}^n |F^{1i} u_1 u_i |]  - \alpha _0 h'|D u|\sum\limits_{i = 1}^n {F^{ii} }  - \alpha _0 ( |D^2 d|  + 1)\sum\limits_{i = 1}^n {F^{ii} }  \notag\\
\ge& \frac{1}{{32(1 + 2M_0 )}}F^{11} w_1 ^2  - C_3 w_1 \sum\limits_{i = 1}^n {F^{ii} }.
\end{align}
From the definition of $w$, we know
\begin{align}\label{5.43}
\frac{{2F^{ij} w_{1ij} }}{{w_1 }}-\frac{{2w_{1t} }}{{w_1 }}
=& \frac{2}{{w_1 }}[(1 +\varphi_u d) F^{ij} u_{ij1}  + (\varphi _{1u} F^{ij} u_{ij}+ \varphi _{uu} u_1 F^{ij} u_{ij})d +\varphi _{u} F^{ij} u_{ij} d_1  ] \notag \\
&-\frac{{2}}{{w_1 }}[(1 +\varphi_u d)u_{1t}  + (\varphi _{1u} u_t+ \varphi _{uu} u_1 u_t)d +\varphi _{u} u_t d_1] \notag \\
&+ \frac{2}{{w_1 }}F^{ij} [d \varphi _{uuu}u_{i}u_{j} u_1  + O(w_1^2)]  \notag  \\
\ge&  -C_4 d  F^{11} {w_1 }^2 - C_4 w_1 \sum\limits_{i = 1}^n {F^{ii} }-C_4.
\end{align}

From \eqref{5.41}, \eqref{5.42} and \eqref{5.43}, we get
\begin{align}\label{5.44}
0 \ge& \sum\limits_{ij = 1}^n {F^{ij} G_{ij} } - G_t \notag \\
\geq& \big[\frac{1}{{32(1 + 2M_0 )}} - C_4 d \big] F^{11} w_1 ^2  - (C_3+C_4)w_1 \sum\limits_{i = 1}^n {F^{ii} } -C_4.
\end{align}

From Lemma \ref{lem2.5}, we know
\begin{align}\label{5.45}
F^{11} \geq c_3 \sum {F^{ii} },
\end{align}
where $c_3 =\frac{{n(k - l)}}{{k(n - l)}}\frac{1}{{n - k + 1}}$. Moreover, from \eqref{2.10} and \eqref{5.39}, we have
\begin{align}\label{5.46}
\sum {F^{ii} } = \sum {\frac{1}{\frac{\sigma _k (D^2 u)}{\sigma _l (D^2 u)}} \frac{\partial [ \frac{\sigma _k (D^2 u)}{\sigma _l (D^2 u)}]} {{\partial u_{ii} }}}=& \frac{1}{f e^{u_t}} \sum { \frac{\partial [ \frac{\sigma _k (D^2 u)}{\sigma _l (D^2 u)}]} {{\partial u_{ii} }}} \notag \\
\geq& \frac{1}{f e^{u_t}} \frac{{k - l}}{k} \frac{1}{C_n^l} (-u_{11})^{k-l-1} \notag \\
\geq& c_5 w_{1}^{2(k-l-1)}.
\end{align}
Then we can get from \eqref{5.44}, \eqref{5.45} and \eqref{5.46}
\begin{align}
 w_1 (x_0, t_0) \leq& C_5. \notag
\end{align}
So we can prove \eqref{5.17}.
\end{proof}

\section{$C^2$-estimate}

We come now to the a priori estimates of global second derivatives, and we obtain the following theorem
\begin{theorem} \label{th6.1}
Suppose $\Omega \subset \mathbb{R}^n$ is a $C^4$ strictly convex domain, and $u \in C^{4}(\overline \Omega \times [0, T))$ is a $k$-admissible solution of equation \eqref{1.1}, satisfying \eqref{1.2} and \eqref{1.3}. Moreover, if $f$ satisfies \eqref{1.4} or $u_0$ satisfies \eqref{1.5}, then we have
\begin{align}\label{6.1}
\sup_{\Omega \times [0, T)} |D^2 u|  \leq M_2,
\end{align}
where $M_2$ depends on $n$, $k$, $l$, $\Omega$, $|D^2 u_0|_{C^0}$, $|Du|_{C^{0}}$, $|u_t|_{C^{0}}$, $\min\limits_{\Omega} f$, $|f|_{C^{2}}$ and $|\varphi|_{C^{3}}$.
\end{theorem}

Following the idea of Lions-Trudinger-Urbas \cite{LTU86} (see also Ma-Qiu \cite{MQ15}), we divide the proof of Theorem \ref{th6.1} into three steps. In step one, we reduce global second derivatives to double normal second derivatives on boundary, then we prove the lower estimate of double normal second derivatives on boundary in step two, and at last we prove the upper estimate of double normal second derivatives on boundary.

\subsection{Global $C^2$ estimates can be reduced to the double normal estimates}
\begin{lemma} \label{lem6.2}
Under the assumptions in Theorem \ref{th6.1}, then we have
\begin{align}\label{6.2}
\sup_{\Omega\times [0, T)} |D^2 u|  \leq C_6(1+ \sup_{\partial \Omega\times [0, T)} |u_{\nu \nu}|),
\end{align}
where $C_6$ depends on $n$, $k$, $l$, $\Omega$, $|D^2 u_0|_{C^0}$, $|Du|_{C^{0}}$, $\min\limits_{\Omega} f$, $|f|_{C^{2}}$, and $|\varphi|_{C^{3}}$.
\end{lemma}
\begin{proof}
Since $\Omega$ is a $C^4$ domain, it is well known that there exists a small positive universal constant $0<{\mu} < \frac{1}{10}$ such that
$d(x) \in C^4(\overline{\Omega_{\mu}})$ and $\nu = - \nabla d$ on $\partial \Omega$. We define $\widetilde{d} \in C^4(\overline{\Omega})$ such that
$\widetilde{d} =d$ in $\overline{\Omega_{\mu}}$ and denote
\begin{align*}
\nu = - \nabla \widetilde{d}, \quad \text{ in } \Omega.
\end{align*}
In fact, $\nu$ is a $C^3(\overline{\Omega})$ extension of the outer unit normal vector field on $\partial \Omega$. We also assume $0 \in \Omega$.

Following the idea of Lions-Trudinger-Urbas \cite{LTU86} (see also Ma-Qiu \cite{MQ15}), we consider the auxiliary function
\begin{align}\label{6.3}
v(x,t,\xi)=u_{\xi\xi}-v'(x,t,\xi)+ |Du|^2+ \frac{K_1}{2}|x|^2,
\end{align}
where $v'(x,t,\xi)=2(\xi\cdot\nu)\xi'\cdot(D\varphi-u_lD\nu^l)=a^lu_l+b$, $\xi'=\xi-(\xi\cdot\nu)\nu$, $a^l=2(\xi\cdot\nu)({\xi'}^l\varphi_u-{\xi'}^iD_i\nu^l)$, $b=2(\xi\cdot\nu){\xi'}^l\varphi_l$, and $K_1$ is a positive universal constant to be determined later.

For any fixed $\xi \in \mathbb{S}^{n-1}$, we have
\begin{align}\label{6.4}
{F^{ij}}{v_{ij}}- v_t =& {F^{ij}}{u_{ij\xi \xi }}-{u_{t\xi \xi }} - \left( {F^{ij}}v{'_{ij}}-{v{'_t}} \right) \notag \\
 &+2{u_k}\left( {F^{ij}}{u_{ijk}}-{{u_{tk}}} \right) + 2{F^{ij}}{u_{ik}}{u_{jk}} + {K_1}\sum\limits_{i = 1}^n {{F^{ii}}},
\end{align}
where $F^{ij}  =: \frac{{\partial [ \log \frac{{\sigma _k (D^2 u )}}{{\sigma _l (D^2 u )}}]}}{{\partial u_{ij} }}$. Direct calculations yield
\begin{align*}
 v' _t  = a^l u_{lt} +  { a^l}_{,t} u_{l} + b_t \geq  a^l u_{lt}  - C_7,
 \end{align*}
and
\begin{align*}
 F^{ij} v' _{ij} =& F^{ij}[a^l u_{lij} +  { a^l}_{,ij} u_l +2  { a^l}_{,i} u_{lj}   + b_{ij} ] \\
  \leq& F^{ij}[a^l u_{lij} +2  { a^l}_{,i} u_{lj} ] + C_{8} \sum\limits_{i = 1}^n {{F^{ii}}}.
 \end{align*}
Hence
\begin{align}\label{6.5}
{F^{ij}}{v_{ij}}- v_t \geq& {F^{ij}}{u_{ij\xi \xi }}-{u_{t\xi \xi }} - a^l ({F^{ij}}{u_{lij}}-{u_{lt }})+2{u_k}( {F^{ij}}{u_{ijk}}-{{u_{tk}}} ) \notag \\
 &- 2  { a^l}_{,i} F^{ij} u_{lj}+ 2{F^{ij}}{u_{ik}}{u_{jk}} +( {K_1} - C_{8})\sum\limits_{i = 1}^n {{F^{ii}}} -C_7\notag \\
\geq& - {F^{ij,kl}}{u_{ij\xi }}{u_{kl\xi }}+\frac{f_u}{f} u_{\xi \xi} - F^{ij}  { a^l}_{,i} { a^l}_{,j} +( {K_1} - C_{8})\sum\limits_{i = 1}^n {{F^{ii}}} -C_{9}\notag \\
\geq&\frac{f_u}{f} v  +( {K_1} - C_{10})\sum\limits_{i = 1}^n {{F^{ii}}} -C_{9}\notag \\
\geq&\frac{f_u}{f} v,
\end{align}
if we choose $K_1$ sufficiently large. So $\max\limits_{\Omega \times [0, T)} v(x, t, \xi)$ attains its maximum on $\partial_p (\Omega \times [0, T])$. Hence $\max\limits_{(\Omega \times [0, T))\times \mathbb{S}^{n-1}} v(x, t, \xi)$ attains its maximum at some point $(x_0, t_0) \in \partial_p (\Omega \times [0, T))$ and some direction $\xi_0 \in \mathbb{S}^{n-1}$. If $t_0 =0$, then $|D^2 u(x_0, t_0)| = |D^2 u_0(x_0)|$, and it is easy to obtain the estimate \eqref{6.2}.

In the following, we assume $t_0 \in (0, T)$.

Case a: $\xi_0$ is tangential to $\partial \Omega$ at $x_0$.

We directly have $\xi_0 \cdot \nu =0$, $v'(x_0, t_0, \xi_0) =0$, and $u_{\xi_0 \xi_0} (x_0, t_0)>0$. In the following, all the calculations are at the point $(x_0, t_0)$ and $\xi = \xi_0$.

From the Neumann boundary condition, we have
\begin{align}\label{6.6}
u_{li}\nu^l =& [c^{ij} + \nu^i \nu^j ]\nu^l u_{lj}\notag  \\
=& c^{ij} [D_j(\nu^l u_{l}) -D_j \nu^l u_l ] + \nu^i \nu^j \nu^l u_{lj} \notag \\
=& c^{ij} D_j\varphi  - c^{ij}u_l D_j \nu^l + \nu^i \nu^j \nu^l u_{lj}.
\end{align}
So it follows that
\begin{align}\label{6.7}
u_{lip} \nu^l =& [c^{pq} + \nu^p \nu^q ] u_{liq} \nu^l \notag \\
=& c^{pq} [D_q (u_{li} \nu^l)- u_{li} D_q \nu^l] + \nu^p \nu^q u_{liq} \nu^l \notag \\
=& c^{pq}D_q (c^{ij} D_j\varphi - c^{ij}u_l D_j \nu^l + \nu^i \nu^j \nu^l u_{lj}) -c^{pq} u_{li} D_q \nu^l + \nu^p \nu^q \nu^l u_{liq}.
\end{align}
then we obtain
\begin{align}\label{6.8}
u_{\xi_0 \xi_0 \nu} =& \sum_{ip l=1}^n  \xi_0^i \xi_0^p u_{lip} \nu^l \notag \\
=& \sum_{ip=1}^n  \xi_0^i \xi_0^p[c^{pq}D_q (c^{ij} D_j\varphi- c^{ij}u_l D_j \nu^l + \nu^i \nu^j \nu^l u_{lj})  -c^{pq} u_{li} D_q \nu^l + \nu^p \nu^q \nu^l u_{liq} ]\notag \\
=& \sum_{i=1}^n  \xi_0^i \xi_0^q[D_q(c^{ij}D_j\varphi -c^{ij}u_l D_j \nu^l + \nu^i \nu^j \nu^l u_{lj}) - u_{li} D_q \nu^l ]\notag \\
=& \xi_0^i \xi_0^q  \left(c^{ij}D_qD_j\varphi + D_qc^{ij}D_j\varphi\right) - \xi_0^j \xi_0^q u_{lq} D_j \nu^l + \xi_0^i \xi_0^q D_q \nu^i u_{\nu \nu}\notag\\
&- \xi_0^i \xi_0^q D_q( c^{ij}D_j \nu^l ) u_l    -\xi_0^i \xi_0^q  u_{li} D_q \nu^l\notag \\
\leq&\varphi_u u_{\xi_0 \xi_0} - 2\xi_0^i  u_{l\xi_0} D_i \nu^l+\xi_0^i \xi_0^q D_q \nu^i u_{\nu \nu}-C_{11}.
\end{align}
 We assume $\xi_0 = e_1$, it is easy to get the bound for $u_{1i}(x_0, t_0)$ for $i > 1$. In fact, we can assume $\xi(\varepsilon) = \frac{(1, \varepsilon, 0, \cdots, 0)}{\sqrt {1+\varepsilon^2}}$. Then we have
\begin{align}\label{6.9}
  0 =& \frac{{dv(x_0, t_0, \xi (\varepsilon))}} {{d\varepsilon}}|_{\varepsilon = 0}  \notag \\
   =& 2u_{ij} (x_0, t_0)\frac{{d\xi ^i (\varepsilon)}} {{d\varepsilon}}|_{\varepsilon = 0} \xi ^j (0) - \frac{{dv'(x_0 ,\xi (\varepsilon))}}
{{d\varepsilon}}|_{\varepsilon = 0}  \notag \\
   =& 2u_{12} (x_0, t_0) - 2\nu ^2 (D _1 \varphi  - u_l D _1 \nu ^l ),
\end{align}
so
\begin{align}\label{6.10}
|u_{12} (x_0, t_0)|= |\nu ^2 (D _1 \varphi - u_l D _1 \nu ^l )| \leq C_{12}.
\end{align}
Similarly, we have for all $i > 1$,
\begin{align}\label{6.11}
|u_{1i} (x_0, t_0)|\leq  C_{12}.
\end{align}
So by the strict convexity of $\Omega$ and $\varphi_u < 0$, we have
\begin{align}\label{6.12}
u_{\xi_0 \xi_0 \nu} \leq  \varphi_u u_{\xi_0 \xi_0} -2 D_1 \nu^1 u_{\xi_0 \xi_0} + C_{13} (1+ | u_{\nu \nu}|)
\leq -2 \kappa_{min} u_{\xi_0 \xi_0} + C_{13} (1+ | u_{\nu \nu}|),
\end{align}
where $\kappa_{min}$ is the minimum principal curvature of $\partial \Omega$ such that $D_1 \nu^1 \geq \kappa_{min} > 0$. Then combining the above with the Hopf lemma, \eqref{6.6} and \eqref{6.11},
\begin{align}\label{6.13}
0 \leq& v_\nu (x_0, t_0, \xi_0)\notag \\
=& u_{\xi_0 \xi_0 \nu} +2u_{i\xi_0}D_j{\xi_0}^i\nu^j-  a^l u_{l\nu}  -  D_\nu a^l u_{l}- b_\nu + 2u_i u_{i\nu}+ K_1 (x \cdot \nu) \notag \\
\leq& u_{\xi_0 \xi_0 \nu} + C_{14} \notag \\
\leq& -2 \kappa_{min} u_{\xi_0 \xi_0} + C_{13} (1+ | u_{\nu \nu}|) + C_{14}.
\end{align}
Then we get
\begin{align}\label{6.14}
u_{\xi_0 \xi_0}(x_0, t_0) \leq  \frac{1} {2\kappa_{min}} (C_{13} + C_{14}) (1+ | u_{\nu \nu}|),
\end{align}
and
\begin{align}\label{6.15}
|u_{\xi \xi}(x, t)| \leq& (n-1) \max\limits_{(\Omega \times [0, T))  \times \mathbb{S}^{n-1}} u_{\xi \xi}(x,t) \notag \\
 \leq& (n-1)[ \max\limits_{(\Omega \times [0, T])  \times \mathbb{S}^{n-1}} v(x, t, \xi)  + C_{15}] = (n-1) [v(x_0, t_0, \xi_0) + C_{15} ]\notag \\
 \leq& (n-1) [u_{\xi_0 \xi_0}(x_0, t_0) +2 C_{15}] \notag \\
\leq& C_{16} (1+ | u_{\nu \nu}|).
\end{align}

Case b: $\xi_0$ is non-tangential.

We  have $\xi_0 \cdot \nu \ne 0$ and  write  $\xi_0 = \alpha \tau + \beta \nu$, where $\tau$ is a tangential vector  and $\alpha = \xi_0 \cdot \tau \geq 0$, $\beta = \xi_0 \cdot \nu \ne 0$, $\alpha ^2 + \beta ^2 =1$ and $\tau \cdot \nu =0$. Then we have

\begin{align}\label{6.16}
u_{\xi_0 \xi_0}(x_0, t_0) =& \alpha^2 u_{\tau \tau}(x_0, t_0) + \beta^2u_{\nu \nu}(x_0, t_0)+ 2 \alpha \beta u_{\tau \nu}(x_0, t_0) \notag \\
=& \alpha^2 u_{\tau \tau}(x_0, t_0) + \beta^2u_{\nu \nu}(x_0, t_0)+ 2  \alpha \beta  [D_i \varphi\tau^i - u_iD_j \nu^i\tau^j]\notag\\
=& \alpha^2 v_{\tau \tau}(x_0, t_0) + \beta^2 v_{\nu \nu}(x_0, t_0)\notag\\
&-|Du|^2-\frac{K_1}{2}|x|^2+ 2  \alpha \beta  (D_i \varphi\tau^i - u_jD_i \nu^j\tau^i)\notag\\
=& \alpha^2 v_{\tau \tau}(x_0, t_0) + \beta^2 v_{\nu \nu}(x_0, t_0)-|Du|^2-\frac{K_1}{2}|x|^2+ v'(x_0, t_0, \xi_0).
\end{align}
Hence
\begin{align}\label{6.17}
v(x_0, t_0, \xi_0) =& \alpha^2 v(x_0, t_0, \tau) + \beta^2 v(x_0, t_0, \nu) \notag \\
\leq&  \alpha^2 v(x_0, t_0, \xi_0)+\beta^2 v(x_0, t_0, \nu),
\end{align}
where the inequality follows from  that $v(x_0, t_0, \xi)$ attains its maximum at the direction $\xi_0$.
Since $\beta\neq0$, we finally obtain
\begin{align}\label{6.18}
v(x_0, t_0, \xi_0)= v(x_0, t_0, \nu).
\end{align}
This yields
\begin{align}\label{6.19}
u_{\xi_0 \xi_0}(x_0, t_0) \leq v(x_0, t_0, \xi_0) + C_{15} = v(x_0, t_0, \nu) +C_{15}  \leq | u_{\nu \nu}|+ 2C_{15}.
\end{align}
Similarly with \eqref{6.15}, we can prove \eqref{6.2}.
\end{proof}

\subsection{Lower estimate of double normal second derivatives on boundary}

\begin{lemma} \label{lem6.3}
Under the assumptions in Theorem \ref{th6.1}, then we have
\begin{align}\label{6.20}
\inf_{\partial \Omega \times [0, T)} u_{\nu \nu}  \geq - C_{17},
\end{align}
where $C_{17}$ is a positive constants depending on $n$, $k$, $l$, $\Omega$, $|u_0|_{C^2}$, $|Du|_{C^{0}}$, $|u_t|_{C^0}$, $\min f$, $|f|_{C^{2}}$ and $|\varphi|_{C^{2}}$.
\end{lemma}

To prove Lemma \ref{lem6.3} and Lemma \ref{lem6.5}, we need the following lemma.
\begin{lemma} \label{lem6.4}
Suppose $\Omega \subset \mathbb{R}^n$ is a $C^2$ strictly convex domain, and $u \in C^{2}(\Omega \times [0, T))$ is a $k$-admissible solution of parabolic Hessian quotient equation \eqref{1.1}. Denote $ F^{ij}  = \frac{{\partial [\log \frac{{\sigma _k (D^2 u )}}{{\sigma _l (D^2 u )}}]}}{{\partial u_{ij} }}$, and
\begin{align}\label{6.21}
h(x) = -d(x) + d^2(x),
\end{align}
where $d(x) =\textrm{dist}(x, \partial \Omega)$ is the distance function of $\Omega$. Then
\begin{align}\label{6.22}
 \sum\limits_{ij=1}^n F^{ij} h_{ij} \geq c_6 (\sum\limits_{i=1}^n F^{ii} + 1), \quad  \text{ in } \Omega_\mu \times [0, T),
\end{align}
where $\Omega_\mu = \{ x \in \Omega: d(x) < \mu\}$ for a small universal constant $\mu$ and $c_6$ is a positive constant depending only on $n$, $k$, $l$, $\Omega$ and $c_2$ (here $c_2$ is defined in \eqref{4.2}).
\end{lemma}

Now we come to prove Lemma \ref{lem6.3}.

\begin{proof}

Firstly, we assume $ \inf_{\partial \Omega \times [0, T)} u_{\nu \nu} < 0$, otherwise there is nothing to prove. Also, if $-\inf_{\partial \Omega\times [0, T)} u_{\nu \nu} < \sup_{\partial \Omega\times [0, T)} u_{\nu \nu}$, that is $\sup_{\partial \Omega\times [0, T)} |u_{\nu \nu}|= \sup_{\partial \Omega\times [0, T)} u_{\nu \nu}$, we can easily get from Lemma \ref{lem6.4}
\begin{align}
- \inf_{\partial \Omega\times [0, T)} u_{\nu \nu}  < \sup_{\partial \Omega\times [0, T)} u_{\nu \nu} \leq C_{21}. \notag
\end{align}
In the following, we assume $- \inf\limits_{\partial \Omega\times [0, T)} u_{\nu \nu} \geq \sup\limits_{\partial \Omega\times [0, T)} u_{\nu \nu}$, that is $\sup\limits_{\partial \Omega\times [0, T)} |u_{\nu \nu}|= -\inf\limits_{\partial \Omega\times [0, T)} u_{\nu \nu}$. For any $T_0 \in (0, T)$, denote $M = - \min\limits_{\partial \Omega\times [0, T_0]} u_{\nu \nu} >0$ and let $(x_1, t_1) \in \partial \Omega \times [0, T_0]$ such that $\min\limits_{\partial \Omega\times [0, T_0]} u_{\nu \nu} =u_{\nu \nu} (x_1, t_1)$.

Now we just need to show that the test function $P(x,t)$ defined below is non-positive in $\Omega_\mu \times [0, T_0]$
\begin{align}\label{6.23}
P(x, t) = [ 1 + \beta d(x)][D u \cdot (-D d)(x) - \varphi (x, u) ] + (A + \frac{1} {2}M)h(x),
\end{align}
where
\begin{align}
\label{6.24} \beta =&\max \{\frac{1}{\mu}, 5 n (2\kappa_{max}+\frac{1}{n}) \frac{C_6}{c_7} \}, \\
\label{6.25} A=& \max\{A_1, A_2, \frac{C_{20} +\frac{2}{n} (k-l) }{c_6} \}.
\end{align}

It is easy to know that $P \leq 0$ on $\partial_p( \Omega_\mu \times [0, T_0])$. Precisely, on $\partial \Omega \times [0, T_0]$, we have $d = h= 0$ and $- D d = \nu$, then
\begin{align}\label{6.26}
P(x, t) =0,  \quad \text{ on } \partial \Omega \times [0, T].
\end{align}
On $(\partial \Omega_\mu \setminus \partial \Omega) \times [0, T_0]$, we have $d = \mu$, then
\begin{align}\label{6.27}
P(x,t) \leq& (1+ \beta \mu) [|D u|  + |\varphi|]  + (A + \frac{1} {2}M) [-\mu + \mu^2]  \notag \\
\leq& C_{18}   - \frac{1}{2} \mu A <0,
\end{align}
since $A \ge 2\mu^{-1}C_{18} =:A_1$. On $ \Omega_\mu \times \{ t =0 \}$, we have $t = 0$, and $0 \leq d \leq \mu$.
For every $x \in \Omega_{\mu}$, there exists $y\in\partial\Omega$ such that $x=y+d(x)\nu(y)$. Thus we have
\begin{align*}
 &Du_0\cdot(-Dd)(x) - \varphi \left( {x,u_0 ( x )} \right) \notag\\
  = &Du_0\cdot(-Dd)(x) - Du_0\cdot(-Dd)(y) + Du_0\cdot(-Dd)(y) - \varphi \left( {x,u_0 ( x )} \right) \notag\\
   = &Du_0\cdot(-Dd)(x) - Du_0\cdot(-Dd)(y) + \varphi \left( {y,u_0 ( y )} \right) - \varphi \left( {x,u_0 ( x )} \right) \notag\\
   =&D\left[Du_0\cdot(-Dd)\right](z)\cdot(x-y)  -(D\varphi)\left(w,u_0(w)\right)\cdot(x-y) \notag\\
   =&d(x)\nu\cdot\{D\left[Du_0\cdot(-Dd)\right](z)\cdot\nu   -(D\varphi)\left(w,u_0(w)\right)\}. \notag
 \end{align*}
 Now  we obtain
 \begin{align}\label{6.28}
\left| Du_0\cdot(-Dd)(x) - \varphi \left( {x,u_0 ( x )} \right) \right| \le C d\left( x \right)  \ \text{in}\  \Omega_{\mu}.
\end{align}
where $C$ is a positive constants depending only on $|u_0|_{C^{2}(\bar\Omega)}$, $|\varphi|_{C^{1}(\bar\Omega)}$ and $\Omega$. Therefore
\begin{align}\label{6.29}
P\left( {x,0} \right) \le&  (1+\beta) \left| Du_0\cdot(-Dd)(x) - \varphi \left( {x,u_0 ( x )} \right)\right| + (A+\frac{M}{2})(-d( x )+d^2(x) )\notag\\
 \le&C(1+\beta)d(x)-\frac{A}{2}d(x) \notag \\
  <& 0,
\end{align}
where we use $A\ge 2C(1+\beta)$.

In the following, we want to show that $P$ attains its maximum only on $\partial \Omega \times [0, T_0]$. Then we can get
\begin{align}
0 \leq P_\nu (x_1, t_1)  =&  [u_{\nu \nu}(x_1, t_1) - \sum\limits_j {u_j d_{j\nu} }     - D_{\nu}\varphi] + (A + \frac{1} {2}M)  \notag\\
\leq& \min_{\partial \Omega \times [0, T_0]} u_{\nu \nu} + C + A + \frac{1} {2}M, \notag
\end{align}
hence \eqref{6.20} holds.

To prove $P$ attains its maximum only on $\partial \Omega \times (0, T_0]$, we assume $P$ attains its maximum at some point $(x_0, t_0) \in \Omega_\mu \times [0, T_0]$ by contradiction. Since $P(x,0)<0$ in $\Omega_\mu$,  we have $t_0>0$.

Rotating the coordinates, we can assume
\begin{align}
D^2 u(x_0, t_0) \text{ is diagonal}. \notag
\end{align}
In the following, all the calculations are at $(x_0, t_0)$. Firstly, we have
\begin{align}
\label{6.30} 0 = P_i =& \beta d_i [-\sum\limits_j{u_j d_j }   - \varphi ] + (1+ \beta d)[- \sum\limits_j{(u_{ji} d_j + u_j d_{ji})}   -\varphi_i-\varphi_u u_i ] + (A + \frac{1} {2}M)h_i\notag \\
 =&\beta d_i [-\sum\limits_j {u_j d_j } - \varphi ] + (1+ \beta d)[-u_{ii} d_i-\sum\limits_j { u_j d_{ji}}   -\varphi_i-\varphi_u u_i ] + (A + \frac{1} {2}M)h_i,  \\
\label{6.31} 0 \leq P_t =&  (1+ \beta d)[-\sum\limits_j u_{jt} d_j - \varphi_u u_t ],
\end{align}
and
\begin{align}\label{6.32}
0 \geq P_{ii}  =& \beta d_{ii} [-\sum\limits_j {u_j d_j } - \varphi] + 2\beta d_i [-\sum\limits_j {(u_{ji} d_j + u_j d_{ji})}   -\varphi_i-\varphi_u u_i]  \notag \\
&+ (1+ \beta d)[-\sum\limits_j {(u_{jii} d_j +2 u_{ji} d_{ji} + u_j d_{jii})} -\varphi_{ii} -2\varphi_{iu}u_i- \varphi_{uu}u_i^2 -\varphi_u u_{ii}] \notag \\
&+ (A + \frac{1} {2}M)h_{ii}   \notag \\
=&  \beta d_{ii} [-\sum\limits_j {u_j d_j } - \varphi] + 2\beta d_i [-u_{ii} d_i-\sum\limits_j {u_j d_{ji}}   -\varphi_i-\varphi_u u_i]  \notag \\
&+ (1+ \beta d)[-\sum\limits_j {u_{jii} d_j }  - 2u_{ii} d_{ii}  - \sum\limits_j {u_j d_{jii} }   -\varphi_{ii} -2\varphi_{iu}u_i- \varphi_{uu}u_i^2 -\varphi_u u_{ii}]  \notag \\
&+ (A + \frac{1} {2}M)h_{ii}   \notag \\
\geq&  - 2\beta u_{ii} d_i ^2  + (1+ \beta d)[-\sum\limits_j {u_{jii} d_j }  - 2u_{ii} d_{ii} -\varphi_u u_{ii}] \notag \\
&+ (A + \frac{1} {2}M)h_{ii} - C_{19},
\end{align}
where $C_{19}$ is a positive constant depending only on $\beta$, $|Du|_{C^0(\bar\Omega)}$, $|\varphi|_{C^2}$ and $\Omega$.

Since $D^2 u(x_0, t_0)$ is diagonal, we know $F^{ij} =0$ for $i \ne j$, where $F^{ij}  =: \frac{{\partial [ \log \frac{{\sigma _k (D^2 u )}}{{\sigma _l (D^2 u )}}]}}{{\partial u_{ij} }}$. Hence
\begin{align}\label{6.33}
0 \geq& \sum\limits_{i = 1}^n {F^{ii} P_{ii} } -P_t \notag \\
\geq&  - 2\beta \sum\limits_{i = 1}^n {F^{ii} u_{ii} d_i ^2 }  + (1+ \beta d)[-\sum\limits_{i,j} {F^{ii} u_{jii} d_j } - 2\sum\limits_{i = 1}^n {F^{ii} u_{ii} d_{ii} } -\varphi_u\sum\limits_{i = 1}^n {F^{ii} u_{ii}}]  \notag \\
&+ (A + \frac{1} {2}M)\sum\limits_{i = 1}^n {F^{ii} h_{ii} }  - C_{19}\sum\limits_{i = 1}^n {F^{ii} } - (1+ \beta d)[-\sum\limits_j u_{jt} d_j + u_t ] \notag \\
\geq&  - 2\beta \sum\limits_{i = 1}^n {F^{ii} u_{ii} d_i ^2 }  - 2(1+ \beta d)\sum\limits_{i = 1}^n {F^{ii} u_{ii} d_{ii} }   \notag \\
&+ [(A + \frac{1} {2}M)c_6  - C_{20}](\sum\limits_{i = 1}^n {F^{ii} }  + 1),
\end{align}
where $C_{20}$ depends only on $n$, $k$, $l$, $\beta$, $\Omega$, $|u_t|_{C^0}$, $|\log f|_{C^1}$, $|Du|_{C^0(\bar\Omega)}$, and $|\varphi|_{C^2}$.

Denote $B = \{ i:\beta d_i ^2  < \frac{1}{n},1 \leq i \leq n\}$ and $G = \{ i:\beta d_i ^2  \geq  \frac{1}{n},1 \leq i \leq n\}$. We choose $\beta \geq \frac{1}{\mu} > 1$, so
\begin{align}\label{6.34}
d_i ^2 < \frac{1}{n} = \frac{1}{n} |D d|^2, \quad i \in B.
\end{align}
It holds $ \sum_{i \in B} d_i ^2 < 1 = |D d|^2$, and $G$ is not empty. Hence for any $i \in G$, it holds
\begin{align}\label{6.35}
d_{i} ^2 \geq \frac{1}{n\beta}.
\end{align}
and from \eqref{6.30}, we have
\begin{align}\label{6.36}
u_{i i }  =  - \frac{{1-2d}}{1+ \beta d}(A + \frac{1}{2}M) + \frac{{\beta \big(-\sum\limits_j {u_j d_j }  - \varphi \big)}}
{1+ \beta d} + \frac{{-\sum\limits_j { u_j d_{ji}}   - D_i\varphi }}{{d_{i } }}.
\end{align}
We choose $A \geq 5 \beta \left(|Du|_{C^0} +  |\varphi|_{C^0}\right) + 5 \sqrt {n\beta} \left(|Du|_{C^0} |D^2d|_{C^0} +  |\varphi|_{C^1}\right) =:A_2$, such that for any $i \in G$
\begin{align}\label{6.37}
\left| \frac{{\beta [-\sum\limits_j {u_j d_j }  - \varphi ]}}{1+ \beta d} + \frac{{-\sum\limits_j { u_j d_{ji}}  -D_i \varphi }}{{d_{i } }}\right|  \leq \frac{A}{5},
\end{align}
then we can get
\begin{align}\label{6.38}
 - \frac{{6A}}{5} - \frac{M}{2} \leqslant u_{i i }  \leq  - \frac{A+M}{5}, \quad \forall \quad i \in G.
\end{align}
Also there is an $i_0  \in G$ such that
\begin{align}\label{6.39}
d_{i_0} ^2 \geq \frac{1}{n} |D d|^2 = \frac{1}{n}.
\end{align}
From \eqref{6.33}, we have
\begin{align}\label{6.40}
0 \geq \sum\limits_{i = 1}^n {F^{ii} P_{ii} } - P_t \geq&  - 2\beta \sum\limits_{i \in G} {F^{ii} u_{ii} d_i ^2 }  - 2\beta \sum\limits_{i \in B} {F^{ii} u_{ii} d_i ^2 }  \notag \\
& - 2(1+ \beta d)\sum\limits_{u_{ii}  > 0} {F^{ii} u_{ii} d_{ii} }  - 2(1+ \beta d)\sum\limits_{u_{ii}  < 0} {F^{ii} u_{ii} d_{ii} }  \notag \\
&+ [(A + \frac{1} {2}M)c_6  - C_{20}](\sum\limits_{i = 1}^n {F^{ii} }  + 1) \notag \\
\geq&  - 2\beta \sum\limits_{i \in G} {F^{ii} u_{ii} d_i ^2 }  - 2\beta \sum\limits_{i \in B} {F^{ii} u_{ii} d_i ^2 }  +  4 \kappa_{max} \sum\limits_{u_{ii}  < 0} {F^{ii} u_{ii} }\notag \\
&+ [(A + \frac{1} {2}M)c_6  - C_{20}](\sum\limits_{i = 1}^n {F^{ii} }  + 1),
\end{align}
where $\kappa_{max}=: \max |D^2 d|$. Direct calculations yield
\begin{align}\label{6.41}
 - 2\beta \sum\limits_{i \in G} {F^{ii} u_{ii} d_i ^2 }  \geq  - 2\beta F^{i_0 i_0 } u_{i_0 i_0 } d_{i_0 } ^2  \geq  - \frac{2\beta }
{{n}}F^{i_0 i_0 } u_{i_0 i_0 },
\end{align}
and
\begin{align}\label{6.42}
- 2\beta \sum\limits_{i \in B} {F^{ii} u_{ii} d_i ^2 }  \geq&  - 2\beta \sum\limits_{i \in B, u_{ii}  > 0} {F^{ii} u_{ii} d_i ^2 }
\geq - \frac{2}{n}  \sum\limits_{i \in B,u_{ii}  > 0} {F^{ii} u_{ii} }  \notag \\
\geq&  - \frac{2}{n}  \sum\limits_{u_{ii}  > 0} {F^{ii} u_{ii} } =  - \frac{2}{n}  [k-l- \sum\limits_{u_{ii}  < 0} {F^{ii} u_{ii} }].
\end{align}

For $u_{i_0 i_0 }<0$, we know from Lemma \ref{lem2.5},
\begin{align}
F^{i_0 i_0 }  \geq c_7 \sum\limits_{i = 1}^n {F^{ii} }.
\end{align}
 So it holds
\begin{align}\label{6.44}
0 \geq \sum\limits_{i = 1}^n {F^{ii} P_{ii} }  \geq& - \frac{2\beta }
{{n}}F^{i_0 i_0 } u_{i_0 i_0 }  +  (4 \kappa_{max}+\frac{2}{n}) \sum\limits_{u_{ii}  < 0} {F^{ii} u_{ii} }  \notag \\
&+ [(A + \frac{1} {2}M)c_5  - C_{20} - \frac{2}{n} (k-l) ](\sum\limits_{i = 1}^n {F^{ii} }  + 1)  \notag \\
\geq& \frac{2\beta } {{n}}c_7  \frac{A+M}{5} \sum\limits_{i = 1}^n {F^{ii} } -  (4\kappa_{max}+\frac{2}{n}) C_6 (1 + M) \sum\limits_{i = 1}^n {F^{ii} }\notag \\
>& 0,
\end{align}
since $\beta \geq 5 n (2\kappa_{max}+\frac{1}{n}) \frac{C_6}{c_7}$. This is a contradiction. So $P$ attains its maximum only on $\partial \Omega  \times (0, T_0]$. The proof of Lemma \ref{lem6.3} is complete.

\end{proof}

\subsection{Upper estimate of double normal second derivatives on boundary}

\begin{lemma} \label{lem6.5}
Under the assumptions in Theorem \ref{th6.1}, then we have
\begin{align}\label{6.45}
\sup_{\partial \Omega \times [0, T)} u_{\nu \nu}  \leq C_{21},
\end{align}
where $C_{21}$ depends on  $n$, $k$, $l$, $\Omega$, $|u_0|_{C^2}$, $|Du|_{C^{0}}$, $|u_t|_{C^0}$, $\min f$, $|f|_{C^{2}}$ and $|\varphi|_{C^{2}}$.
\end{lemma}

\begin{proof}

Firstly, we assume $ \sup_{\partial \Omega \times [0, T)} u_{\nu \nu} > 0$, otherwise there is nothing to prove. Also, if $\sup_{\partial \Omega \times [0, T)} u_{\nu \nu} < - \inf_{\partial \Omega \times [0, T)} u_{\nu \nu}$, that is $\sup_{\partial \Omega \times [0, T)} |u_{\nu \nu}|= -\inf_{\partial \Omega \times [0, T)} u_{\nu \nu}$, we can easily get from Lemma \ref{lem6.3}
\begin{align}\label{6.46}
\sup_{\partial \Omega \times [0, T)} u_{\nu \nu} < - \inf_{\partial \Omega \times [0, T)} u_{\nu \nu} \leq C_{17}.
\end{align}
In the following, we assume $\sup_{\partial \Omega \times [0, T)} u_{\nu \nu} \geq - \inf_{\partial \Omega \times [0, T)} u_{\nu \nu}$, that is $\sup_{\partial \Omega \times [0, T)} |u_{\nu \nu}|= \sup_{\partial \Omega \times [0, T)} u_{\nu \nu}$. For any $T_0 \in (0, T)$, denote $M = \max_{\partial \Omega \times [0, T_0]} u_{\nu \nu} >0$ and let $(\widetilde{x_1}, \widetilde{t_1}) \in \partial \Omega \times [0, T_0]$ such that $\max_{\partial \Omega \times [0, T_0]} u_{\nu \nu} =u_{\nu \nu} (\widetilde{x_1}, \widetilde{t_1})$.

Now we consider the test function
\begin{align}\label{6.47}
\widetilde{P}(x,t) = [ 1 + \beta d(x)][D u \cdot (-D d)(x) - \varphi (x, u) ] - (A + \frac{1} {2}M)h(x),
\end{align}
where
\begin{align*}
\beta =&\max \{\frac{1}{\mu}, \frac{5 n}{2} (2\kappa_{max}+\frac{1}{n}) \frac{C_6}{c_1} \}, \\
A=& \max\{A_1, A_2, A_3, A_4, \frac{C_{23}}{c_6} \}.
\end{align*}

Similarly, we first show that $\widetilde{P} \geq 0$ on $\partial_p( \Omega_\mu \times [0, T_0])$. Precisely, on $\partial \Omega \times [0, T_0]$, we have
\begin{align}\label{6.48}
\widetilde{P}(x, t) =0.
\end{align}
On $(\partial \Omega_\mu \setminus \partial \Omega) \times [0, T_0]$, we have $d = \mu$, and then
\begin{align}\label{6.49}
\widetilde{P}(x,t) \geq& -(1+ \beta \mu) \left(|D u|  +  |\varphi|\right)  + (A + \frac{1} {2}M) [\mu - \mu^2]  \notag \\
\geq&- C_{18}   + \frac{1}{2} \mu A >0,
\end{align}
since $A \ge 2\mu^{-1}C_{18} =:A_1$. On $ \Omega_\mu \times \{ t =0 \}$, we have from \eqref{6.28}
\begin{align}
\widetilde{P}\left( {x,0} \right) \ge&  -(1+\beta) \left|  Du_0\cdot(-Dd)(x) - \varphi (x,u_0 ( x )) \right| + (A+\frac{M}{2})(d( x )-d^2(x) )\notag\\
 \ge& - C(1+\beta)d(x) +\frac{A}{2}d(x)\notag\\
  >& 0,
\end{align}
where we used $A\ge 2C(1+\beta)$.

In the following, we want to prove $\widetilde{P}$ attains its minimum only on $\partial \Omega \times [0, T_0]$. Then we can get
\begin{align}\label{6.51}
0 \geq \widetilde{P}_\nu (\widetilde{x_1}, \widetilde{t_1})  =&  [u_{\nu \nu}(\widetilde{x_1}, \widetilde{t_1}) - \sum\limits_j {u_j d_{j\nu} }     - D_{\nu}\varphi] - (A + \frac{1} {2}M)  \notag\\
\geq& \max_{\partial \Omega} u_{\nu \nu} - C - A  - \frac{1} {2}M,
\end{align}
hence \eqref{6.45} holds.

To prove $\widetilde{P}$ attains its minimum only on $\partial \Omega \times [0, T_0]$, we assume $\widetilde{P}$ attains its minimum at some point  $(\widetilde{x_0}, \widetilde{t_0}) \in \Omega_\mu \times (0, T_0]$ by contradiction.

 Rotating the coordinates, we can assume
\begin{align}\label{6.52}
D^2 u(\widetilde{x_0}, \widetilde{t_0}) \text{ is diagonal}.
\end{align}
In the following, all the calculations are at $(\widetilde{x_0}, \widetilde{t_0})$.

Firstly, we have
\begin{align}
\label{6.53} 0 = \widetilde{P}_i =& \beta d_i [-\sum\limits_j {u_j d_j }  - \varphi ] + (1+ \beta d)[-\sum\limits_j {(u_{ji} d_j + u_j d_{ji})}   -\varphi_i- \varphi_u u_i ] - (A + \frac{1} {2}M)h_i\notag \\
 =&\beta d_i [-\sum\limits_j {u_j d_j }  - \varphi ] + (1+ \beta d)[-u_{ii} d_i-\sum\limits_j { u_j d_{ji}}   -\varphi_i- \varphi_u u_i ] - (A + \frac{1} {2}M)h_i,\\
\label{6.54} 0 \geq \widetilde{P}_t =&  (1+ \beta d)[-\sum\limits_j u_{jt} d_j +  \varphi_u u_t ],
\end{align}
and
\begin{align}\label{6.55}
0 \leq \widetilde{P}_{ii}  =& \beta d_{ii} [-\sum\limits_j {u_j d_j } - \varphi] + 2\beta d_i [-\sum\limits_j {(u_{ji} d_j + u_j d_{ji})}   -\varphi_i- \varphi_u u_i]  \notag \\
&+ (1+ \beta d)[-\sum\limits_j {(u_{jii} d_j +2 u_{ji} d_{ji} + u_j d_{jii})} - \varphi_{ii} - 2\varphi_{iu}u_i-\varphi_{uu}u_i^2-\varphi_{u}u_{ii} ]\notag \\
& - (A + \frac{1} {2}M)h_{ii}   \notag \\
=&  \beta d_{ii} [-\sum\limits_j {u_j d_j } + u- \varphi] + 2\beta d_i [-u_{ii} d_i-\sum\limits_j {u_j d_{ji}}   + u_i- \varphi_i]  \notag \\
&+ (1+ \beta d)[-\sum\limits_j {u_{jii} d_j }  - 2u_{ii} d_{ii}  - \sum\limits_j {u_j d_{jii} }  - \varphi_{ii} - 2\varphi_{iu}u_i-\varphi_{uu}u_i^2-\varphi_{u}u_{ii}]  \notag \\
&- (A + \frac{1} {2}M)h_{ii}   \notag \\
\leq&  - 2\beta u_{ii} d_i ^2  + (1+ \beta d)[-\sum\limits_j {u_{jii} d_j }  - 2u_{ii} d_{ii} -\varphi_{u} u_{ii}]- (A + \frac{1} {2}M)h_{ii} + C_{22}.
\end{align}

Since $D^2 u(\widetilde{x_0}, \widetilde{t_0})$ is diagonal, we know $F^{ij} =0$ for $i \ne j$. Hence
\begin{align}\label{6.56}
0 \leq& \sum\limits_{i = 1}^n {F^{ii} \widetilde{P}_{ii} }  - \widetilde{P}_t\notag \\
\leq&  - 2\beta \sum\limits_{i = 1}^n {F^{ii} u_{ii} d_i ^2 }  + (1+ \beta d)\big[-\sum\limits_{i,j} {F^{ii} u_{jii} d_j }  - 2\sum\limits_{i = 1}^n {F^{ii} u_{ii} d_{ii} } -\varphi_{u}\sum\limits_{i = 1}^n {F^{ii} u_{ii} }\big]  \notag \\
&- (A + \frac{1} {2}M)\sum\limits_{i = 1}^n {F^{ii} h_{ii} }  + C_{22}\sum\limits_{i = 1}^n {F^{ii} }  -(1+ \beta d)[-\sum\limits_j u_{jt} d_j - \varphi_u u_t] \notag \\
\leq&  - 2\beta \sum\limits_{i = 1}^n {F^{ii} u_{ii} d_i ^2 }  - 2(1+ \beta d)\sum\limits_{i = 1}^n {F^{ii} u_{ii} d_{ii} }   \notag \\
&+ [-(A + \frac{1} {2}M)c_6  + C_{23}](\sum\limits_{i = 1}^n {F^{ii} }  + 1).
\end{align}

Denote $B = \{ i:\beta d_i ^2  < \frac{1}{n},1 \leq i \leq n\}$ and $G = \{ i:\beta d_i ^2  \geq  \frac{1}{n},1 \leq i \leq n\}$. We choose $\beta \geq \frac{1}{\mu} >1$, so
\begin{align}\label{6.57}
d_i ^2 < \frac{1}{n} = \frac{1}{n} |D d|^2, \quad i \in B.
\end{align}
It holds $ \sum_{i \in B} d_i ^2 < 1 = |D d|^2$, and $G$ is not empty. Hence for any $i \in G$, it holds
\begin{align}\label{6.58}
d_{i} ^2 \geq \frac{1}{n\beta}.
\end{align}
and from \eqref{6.53}, we have
\begin{align}\label{6.59}
u_{i i }  =  \frac{{1-2d}}{1+ \beta d}(A + \frac{1}{2}M) + \frac{{\beta [-\sum\limits_j {u_j d_j }  - \varphi ]}}
{1+ \beta d} + \frac{{-\sum\limits_j { u_j d_{ji}}-D_i \varphi }}{{d_{i } }}.
\end{align}
We choose $A \geq 5 \beta \left(|Du|_{C^0} +  |\varphi|_{C^0}\right) + 5 \sqrt {n\beta} \left(|Du|_{C^0} |D^2d|_{C^0} +  |\varphi|_{C^1}\right) =:A_2$, such that for any $i \in G$
\begin{align}\label{6.60}
| \frac{{\beta [-\sum\limits_j {u_j d_j }  - \varphi ]}}{1+ \beta d} + \frac{{-\sum\limits_j { u_j d_{ji}}- D_i \varphi }}{{d_{i } }}|
\leq \frac{A}{5},
\end{align}
then we can get
\begin{align}\label{6.61}
\frac{3A}{5}+ \frac{2M}{5} \leqslant u_{i i }  \leq \frac{6A}{5}+ \frac{M}{2}, \quad \forall \quad i \in G.
\end{align}
Also there is an $i_0  \in G$ such that
\begin{align}\label{6.62}
d_{i_0} ^2 \geq \frac{1}{n} |D d|^2 = \frac{1}{n}.
\end{align}
From \eqref{6.56}, we have
\begin{align}\label{6.63}
0 \leq \sum\limits_{i = 1}^n {F^{ii} \widetilde{P}_{ii} }  -\widetilde{P}_t\leq&  - 2\beta \sum\limits_{i \in G} {F^{ii} u_{ii} d_i ^2 }  - 2\beta \sum\limits_{i \in B} {F^{ii} u_{ii} d_i ^2 }  \notag \\
&- 2(1+ \beta d)\sum\limits_{u_{ii}  > 0} {F^{ii} u_{ii} d_{ii} }  - 2(1+ \beta d)\sum\limits_{u_{ii}  < 0} {F^{ii} u_{ii} d_{ii} }  \notag \\
&+ [-(A + \frac{1} {2}M)c_6  + C_{23}](\sum\limits_{i = 1}^n {F^{ii} }  + 1) \notag \\
\leq&  - 2\beta \sum\limits_{i \in G} {F^{ii} u_{ii} d_i ^2 }  - 2\beta \sum\limits_{i \in B} {F^{ii} u_{ii} d_i ^2 }  \notag \\
&+4 \kappa_{max} \sum\limits_{u_{ii}  > 0} {F^{ii} u_{ii} } \notag \\
&+ [-(A + \frac{1} {2}M)c_6 + C_{23}](\sum\limits_{i = 1}^n {F^{ii} }  + 1),
\end{align}
where $\kappa_{max}=: \max |D^2 d|$. Direct calculations yield
\begin{align}\label{6.64}
 - 2\beta \sum\limits_{i \in G} {F^{ii} u_{ii} d_i ^2 }  \leq  - 2\beta F^{i_0 i_0 } u_{i_0 i_0 } d_{i_0 } ^2  \leq  - \frac{2\beta }
{{n}}F^{i_0 i_0 } u_{i_0 i_0 },
\end{align}
and
\begin{align}\label{6.65}
- 2\beta \sum\limits_{i \in B} {F^{ii} u_{ii} d_i ^2 }  \leq&  - 2\beta \sum\limits_{i \in B, u_{ii}  < 0} {F^{ii} u_{ii} d_i ^2 }
\leq - \frac{2}{n}\sum\limits_{i \in B,u_{ii} < 0} {F^{ii} u_{ii} }  \notag \\
\leq&  - \frac{2}{n}\sum\limits_{u_{ii} < 0} {F^{ii} u_{ii} } \notag \\
=&- \frac{2}{n}[k-l - \sum\limits_{u_{ii} > 0} {F^{ii} u_{ii} }].
\end{align}

So it holds
\begin{align}\label{6.66}
0 \leq& \sum\limits_{i = 1}^n {F^{ii} \widetilde{P}_{ii} } - \widetilde{P}_t  \notag \\
\leq& - \frac{2\beta } {{n}}F^{i_0 i_0 } u_{i_0 i_0 }  +  (4\kappa_{max}+ \frac{2}{n}) \sum\limits_{u_{ii}  > 0} {F^{ii} u_{ii} } + [- (A + \frac{1} {2}M)c_6 + C_{23}](\sum\limits_{i = 1}^n {F^{ii} }  + 1).
\end{align}
We divide into three cases to prove the result. Without loss of generality, we assume that $i_0=1 \in G$, and $u_{22} \geq \cdots \geq u_{nn}$.

$\blacklozenge$ CASE I: $u_{nn} \geq 0$.

In this case, we have
\begin{align}\label{6.67}
(4\kappa_{max}+ \frac{2}{n}) \sum\limits_{u_{ii}  > 0} {F^{ii} u_{ii} } =  (4\kappa_{max}+ \frac{2}{n}) \sum\limits_{i =1}^n {F^{ii} u_{ii} }=  (4\kappa_{max}+ \frac{2}{n})(k-l).
\end{align}
Hence from \eqref{6.66} and \eqref{6.67}
\begin{align}\label{6.68}
0 \leq \sum\limits_{i = 1}^n {F^{ii} \widetilde{P}_{ii} } - \widetilde{P}_t  \leq& (4\kappa_{max}+ \frac{2}{n}) \sum\limits_{u_{ii}  > 0} {F^{ii} u_{ii} }+ [- (A + \frac{1} {2}M)c_6 + C_{23}](\sum\limits_{i = 1}^n {F^{ii} }  + 1) \notag \\
\leq&  k(4\kappa_{max}+ 1)  + \left[- (A + \frac{1} {2}M)c_6 + C_{23}\right] \notag \\
<& 0,
\end{align}
since $A \geq \frac{  k(4\kappa_{max}+ 1)  +C_{23}}{c_6} =: A_3$. This is a contradiction.

$\blacklozenge$ CASE II: $u_{nn} < 0$ and $-u_{nn} < \frac{c_6}{10 (4\kappa_{max}+ \frac{2}{n})}u_{11}$.

In this case, we have
\begin{align}\label{6.69}
(4\kappa_{max}+ \frac{2}{n}) \sum\limits_{u_{ii}  > 0} {F^{ii} u_{ii} } =& (8\kappa_{max}+ \frac{2}{n}) [ k-l  - \sum\limits_{u_{ii} < 0} {F^{ii} u_{ii} }] \notag  \\
\leq& (4\kappa_{max}+ \frac{2}{n}) [ k-l - u_{nn}\sum\limits_{i=1}^n {F^{ii} }] \notag  \\
<& k(4\kappa_{max}+ 1)   +  \frac{c_6}{10}u_{11} \sum\limits_{i=1}^n {F^{ii} } \notag  \\
\leq& k(4\kappa_{max}+ 1) + \frac{c_6}{10} (\frac{6A}{5}+ \frac{M}{2} )\sum\limits_{i=1}^n {F^{ii} }.
\end{align}
Hence from \eqref{6.66} and \eqref{6.69}
\begin{align}\label{6.70}
0 \leq& \sum\limits_{i = 1}^n {F^{ii} \widetilde{P}_{ii} } - \widetilde{P}_t \notag \\
\leq&  (4\kappa_{max}+ \frac{2}{n}) \sum\limits_{u_{ii}  > 0} {F^{ii} u_{ii} } + \big[- (A + \frac{1} {2}M)c_6 + C_{23}\big]\big(\sum\limits_{i = 1}^n {F^{ii} }  + 1\big) \notag \\
<&  k(4\kappa_{max}+ 1)  + \frac{c_6}{10} \big(\frac{6A}{5}+ \frac{M}{2} \big)\sum\limits_{i=1}^n {F^{ii} } + \big[- (A + \frac{1} {2}M)c_6 + C_{23}\big]\big(\sum\limits_{i = 1}^n {F^{ii} }  + 1\big) \notag \\
<&0,
\end{align}
since $A \geq \max\{\frac{  k(4\kappa_{max}+ 1)  +C_{23}}{c_6}, \frac{25 C_{23}}{3c_6} \} =: A_4$. This is a contradiction.

$\blacklozenge$ CASE III: $u_{nn} < 0$ and $-u_{nn} \geq \frac{c_6}{10 (4\kappa_{max}+ \frac{2}{n})} u_{11} $.

In this case, we have $u_{11} \geq \frac{3A}{5}+ \frac{2M}{5}$, and $u_{22} \leq C_6 (1+M)$. So
\begin{align} \label{6.71}
u_{11} \geq \frac{2}{5C_6} u_{22}.
\end{align}
Let $\delta = \frac{2}{5C_6}$ and $\varepsilon =\frac{c_6}{10 (4\kappa_{max}+ \frac{2}{n})}$, \eqref{2.12} in Lemma \ref{lem2.6} holds, that is
\begin{align} \label{6.72}
F^{11} \geq c_1 \sum\limits_{i = 1}^n {F^{ii} },
\end{align}
where $c_1 = \frac{n} {k}\frac{{k - l}}{{n - l}}\frac{c_0^2} {{n - k + 1}}$ and  $c_0 = \min \{ \frac{{\varepsilon ^2 \delta ^2}}{{2(n - 2)(n - 1)}},\frac{{\varepsilon ^2 \delta }}{{4(n- 1)}}\}$. Hence from \eqref{6.66} and \eqref{6.72}
\begin{align}\label{6.73}
0 \leq \sum\limits_{i = 1}^n {F^{ii} \widetilde{P}_{ii} }  \leq& - \frac{2\beta }
{{n}}F^{11 } u_{11}  +  (4\kappa_{max}+ \frac{2}{n}) \sum\limits_{u_{ii}  > 0} {F^{ii} u_{ii} } \notag \\
&+ \left[- (A + \frac{1} {2}M)c_6 + C_{23}\right]\left(\sum\limits_{i = 1}^n {F^{ii} }  + 1\right)\notag \\
\leq& - \frac{2\beta } {{n}}c_1  (\frac{3A}{5}+ \frac{2M}{5}) \sum\limits_{i = 1}^n {F^{ii} } + (4\kappa_{max}+1) C_6(1 + M) \sum\limits_{i = 1}^n {F^{ii} }\notag \\
<& 0,
\end{align}
since $\beta \geq 5n(2\kappa_{max}+1) \frac{C_6}{c_1}$. This is a contradiction.

So $\widetilde{P}(x,t)$ attains its maximum only on $\partial \Omega\times [0, T_0]$. The proof of Lemma \ref{lem6.5} is complete.

\end{proof}

\section{Proof of Theorem 1.1}

For the Neumann problem of parabolic Hessian quotient equations \eqref{1.1}, we have established the $|u_t|$, $C^0$, $C^1$ and $C^2$ estimates in Section 3, Section 4, Section 5, Section 6, respectively. Then the equation \eqref{1.1} is uniformly parabolic in $\overline \Omega \times [0, T)$.
Due to the concavity of operator $\log \big(\frac{\sigma _k (\lambda)}{\sigma _l (\lambda)} \big)$ in $\Gamma_k$, we can get the global H\"{o}lder estimates of second derivative following the discussions in \cite{LT86}, the uniform estimates of all higher derivatives of $u$ can be derived by differentiating the equation \eqref{1.1} and apply the Schauder theory for linear, uniformly parabolic equations. Applying the method of continuity (see \cite{GT}, Theorem 17.28), we can get the existence of smooth $k$-admissible solution $u(x, t)$.

By the uniform estimates of $u$ and the uniform parabolicity of equation \eqref{1.1}, the solution $u(x,t)$ exists for all time $t \geq 0$, that is $T = +\infty$.

Following the discussions in \cite{SS03}, we can obtain the smooth convergence of $u(x,t)$. That is,
\begin{align}
\lim\limits_{t \rightarrow +\infty}u ( {x,t} )=u^\infty (x),
\end{align}
and $u^\infty (x)$ satisfies the equation \eqref{1.7}.

If $f$ satisfies $\eqref{1.4}$, we know from \eqref{3.2}
\begin{align}
|u_t \left( {x,t} \right)| \leq C_{24} e^{-c_f t}.
\end{align}
Hence the rate of convergence is exponential.

\section{Proof of Theorem 1.2}

In this section, we prove Theorem \ref{th1.2}, following the ideas of Schn\"{u}rer-Schwetlick \cite{SS04}, Qiu-Xia \cite{QiuXia16} and Ma-Wang-Wei \cite{MWW18}.

\subsection{elliptic problem}

Firstly, we solve the following elliptic problem, which is the key of proof of Theorem \ref{th1.2}.

\begin{theorem}\label{th8.1}
Let $\Omega$ is a strictly convex bounded domain in $\mathbb{R}^n$ with smooth boundary. Assume that $u_0$ is given as in Theorem \ref{th1.2} and $f$ is a positive smooth function, $f\in C^{\infty}(\overline\Omega)$.
Then there exists a unique $s\in\mathbb{R}$ and a $k$-convex function $u\in C^{\infty}(\overline\Omega)$ solving
\begin{align} \label{8.1}
\left\{ {\begin{array}{*{20}c}
   {\frac{\sigma _k( {D^2 u})}{\sigma _l( {D^2 u})} = f\left( x \right)e^s, \quad x \in \Omega,}  \\
   {\frac{{\partial u}}{{\partial \nu}} = \varphi \left( x \right), \qquad \qquad x \in \partial \Omega.}  \\
\end{array}} \right.
\end{align}
Moreover, the solution $u$ is unique up to a constant.
\end{theorem}

 \begin{proof}

To find a pair $(s,u)$ solving the above equation, we consider the following approximating equation
\[
\left( { * _{\varepsilon ,s} } \right)\left\{ \begin{array}{*{20}c}
   {\frac{\sigma _k( {D^2 u})}{\sigma _l( {D^2 u})} = f\left( x \right)e^{s + \varepsilon u},\quad x \in \Omega,}  \\
   {\frac{{\partial u}}{{\partial \nu}} = \varphi \left( x \right), \quad \qquad \qquad x \in \partial \Omega.}  \\
\end{array} \right.
\]

Let $u_{\varepsilon,s}(x)$ be the $k$-admissible solution of ($* _{\varepsilon ,s}$) if the solution exists, then we have
\[u_{\varepsilon,s}(x)=u_{\varepsilon,0}(x)-\frac{s}{\varepsilon}.\]
Thus $u_{\varepsilon,s}(x)$ is strictly decreasing with respect to $s$.\\

In the following, we will prove that for any $\varepsilon>0$, there exists a unique constant $s_\varepsilon$ which is uniformly bounded such that $|u_{\varepsilon,s_\varepsilon}|_{C^k(\bar\Omega)}$ ($k$ is any positive integer) is uniformly bounded. Thus by extracting subsequence, we have $s_{\varepsilon_i}$ converges to $s$ and $u_{\varepsilon,s_\varepsilon}$ converges to a solution $u$ of our problem \eqref{8.1}.

\textbf{Step 1:} If we choose $M$ sufficiently large, we have that $ u_\varepsilon^+ = u_0+\frac{M}{\varepsilon}$ is a supersolution of $(*_{\varepsilon,0})$ and $ u_\varepsilon^-=u_0-\frac{M}{\varepsilon}$ is a subsolution of $(*_{\varepsilon,0})$, i.e. $u_\varepsilon ^- \le u_{\varepsilon, 0}\leq u_\varepsilon ^ +$. Indeed, we have
\begin{align}\label{8.2}
 &\frac{\sigma _k}{\sigma _l} \left( {D^2 u_{\varepsilon ,0} } \right)e^{ - \varepsilon u_{\varepsilon ,0} }  -\frac{\sigma _k}{\sigma _l} \left( {D^2 u_\varepsilon ^ +  } \right)e^{ - \varepsilon u_\varepsilon ^ +  } \notag \\
  =& \int_0^1 {\frac{d}{{dt}}} \left[ {\frac{\sigma _k}{\sigma _l}  \left( {D^2 \left( {tu_{\varepsilon ,0}  + (1 - t)u_\varepsilon^ + } \right)} \right)e^{ - \varepsilon \left( {tu_{\varepsilon ,0}  + (1 - t) u_\varepsilon ^ + } \right)} } \right]dt \notag\\
  =& a^{ij} \left( {u_{\varepsilon ,0}  -  u_\varepsilon ^ + } \right)_{ij}  - c\left( {u_{\varepsilon ,0}  -  u_\varepsilon ^ + } \right),
\end{align}
where
\[
a^{ij}= \int_0^1 {\frac{\partial \big(\frac{\sigma _k}{\sigma _l} ( {D^2 \left( {tu_{\varepsilon ,0}  + (1 - t) u_\varepsilon ^ + } \right)} ) \big)}{\partial u_{ij}}e^{ - \varepsilon \left( {tu_{\varepsilon ,0}  + (1 - t) u_\varepsilon ^ + } \right)} } dt,
\]
is positive definite and
\[
c=\varepsilon \int_0^1 {\frac{\sigma _k}{\sigma _l}  \left( {D^2 \left( {tu_{\varepsilon ,0}  + (1 - t) u_\varepsilon ^ + } \right)} \right)e^{ - \varepsilon \left( {tu_{\varepsilon ,0}  + (1 - t) u_\varepsilon ^ + } \right)} } dt>0.
\]

On the other hand, by the equation
\begin{align}\label{8.3}
 &\frac{\sigma _k}{\sigma _l}  \left( {D^2 u_{\varepsilon ,0} } \right)e^{ - \varepsilon u_{\varepsilon ,0} }  - \frac{\sigma _k}{\sigma _l}  \left( {D^2 u_\varepsilon ^ +  } \right)e^{ - \varepsilon u_\varepsilon ^ +  } \notag \\
  =& f\left( x \right) - \frac{\sigma _k}{\sigma _l}  ( {D^2 u_0 } ) e^{ - M - \varepsilon u_0 } \notag\\
  =& f( x )\left( {1 -\frac{1}{f(x)} e^{ - M - \varepsilon u_0 } \frac{\sigma _k}{\sigma _l}  ( {D^2 u_0 } )} \right) > 0,
 \end{align}
where we choose $M=1+\max |\log f|+ \max|u_0|+\max{\big|\log{\big(\frac{\sigma _k({D^2 u_0})}{\sigma _l({D^2 u_0})} \big)}\big|}$ and $\varepsilon<1$.

Combining \eqref{8.2} with \eqref{8.3}, we obtain
\begin{align*}
\left\{ \begin{array}{*{20}c}
   {a^{ij} \left( {u_{\varepsilon ,0}  - u_\varepsilon ^ +  } \right)_{ij}  - c\left( {u_{\varepsilon ,0}  - u_\varepsilon ^ +  } \right) > 0,\quad x \in \Omega, }  \\
   {\frac{{\partial \left( {u_{\varepsilon ,0}  - u_\varepsilon ^ +  } \right)}}{{\partial \nu }} = \frac{{\partial \left( {u_{\varepsilon ,0}  - u_0 } \right)}}{{\partial \nu }} = 0, \qquad \qquad x \in \partial \Omega.}  \\
\end{array} \right.
 \end{align*}
The maximum principle yields that $u_{\varepsilon ,0}<u_\varepsilon ^ +$ in $\Omega$. Similarly, $u_{\varepsilon ,0}>u_\varepsilon ^ -$ in $\Omega$. Thus we have $u_{\varepsilon, M}< u_0< u_{\varepsilon, -M}$ in $\Omega$.  By strictly decreasing property of $u_{\varepsilon, s}$, for any $\varepsilon\in(0,1)$, there exists a  unique $s_{\varepsilon}\in (-M, M)$ such that $u_{\varepsilon, s_{\varepsilon}}(y_0) =u_0(y_0)$ for a fixed point $y_0 \in \Omega$.  We also have $|\varepsilon u_{\varepsilon,s_{\varepsilon}}|_{C^0(\bar\Omega)} \le 2M + \varepsilon |u_0|_{C^0(\bar\Omega)} \le 3M$.

\textbf{Step 2:} We prove that for $\varepsilon>0$ sufficiently small, $|Du_{\varepsilon,s_{\varepsilon}}|\le C_{25}$, where $C_{25}$ is a positive constant independent of $|u_{\varepsilon,s_{\varepsilon}}|_{C^0(\bar\Omega)}$.

We denote $\widetilde{F}^{ij}  =: \frac{\partial \big(\frac{\sigma _k(D^2u)}{\sigma _l(D^2u)}  \big)}{\partial u_{ij}}$. By the equation, we have
\begin{align} \label{8.4}
\sum\limits_{i = 1}^n {\widetilde{F}^{ii} }  =& \frac{{(n - k + 1)\sigma _{k - 1}\sigma _l - (n - l + 1)\sigma _k \sigma _{l - 1} }} {{\sigma _l ^2 }}\geq \frac{{k - l}}{k}(n - k + 1) \frac{{\sigma _{k - 1} (\lambda )}}{{\sigma _l (\lambda )}}  \notag \\
\ge& c(n,k)\big(\frac{\sigma _k(D^2u)}{\sigma _l(D^2u)}  \big)^{ \frac{k-l-1}{k-l}}  = c( {n,k} )f^{  \frac{k-l-1}{k-l}} e^{ \frac{k-l-1}{k-l} ( {s_\varepsilon   + \varepsilon u})}  \ge c_8 >0.
\end{align}

We use the following auxiliary function
\begin{align*}
 G = \log \left| {Dw} \right|^2  + ah,
\end{align*}
where $w= u - \varphi h$, $h$ is the defining function with $|Dh|^2\le \kappa_1$ and ${D^2h}\ge\kappa_2I$, and $a=\min\{2 \kappa_2, \frac{\kappa_2}{\kappa_1} \}$. Suppose that $G$ attains its maximum at the point $x_0$. We claim that $x_0\in \Omega$. In fact, $x_0\in\partial\Omega$, we assume $x_0 =0$ and choose the coordinate such that $\partial\Omega\cap B_{\delta}(x_0)$ can be represented as $(x', x_n)$ with $x_n=\rho(x')$, where $\rho(x')$ satisfies $\rho(x_0')=0$ and $D_{x'}\rho(x_0')=0$. Also we have $\nu(x_0)=(0,\cdots,0,-1) = Dh(x_0)$, and then $w_n(x_0) = u_n -\varphi \cdot (-1) = -u_\nu +\varphi =0$. Rotating the $x'$-axis, we can further assume that $w_1(x_0)=|Dw|(x_0)$. Moreover we have $u_1(x_0)=w_1(x_0)$ and $u_i(x_0)=w_i(x_0)=0$ for $2\le i\le n-1$.

By Hopf lemma, we can get
\begin{align*}
0>-\frac{\partial G}{\partial \nu}(x_0)=&\frac{\partial G}{\partial x_n}(x_0)\\
=&\frac{2w_kw_{kn}}{|Dw|^2}+ah_n
=\frac{2w_{1n}}{w_1} - a\\
=&\frac{2u_{1n}}{w_1} - \frac{2\varphi_{1}h_n+2\varphi_{n}h_1+2\varphi h_{1n}}{w_1} - a\\
=&\frac{2u_{1n}}{w_1}-\frac{-2\varphi_{1}+2\varphi h_{1n}}{w_1} - a\\
=&2{\nu^1}_1 - a \geq 2 \kappa_2 - a \geq 0,
\end{align*}
where we have used the equality $u_{n1}=u_{k}{\nu^k}_1-\varphi_{1}=u_{1}{\nu^1}_1+\varphi{\nu^n}_1-\varphi_{1}$. Contradiciton.

Hence $x_0\in\Omega$, and then we have
\begin{align*}
0 = G_i  = \frac{{2w_k w_{ki} }}{{\left| {Dw} \right|^2 }} + ah_i  = \frac{{2w_{1i} }}{{w_1 }} + ah_i, \quad i=1, \cdots, n.
\end{align*}
Hence
\begin{align*}
\frac{{w_{1i} }}{{w_1 }} =  - \frac{1}{2}ah_{i}, \quad  w_{1i}  =  - \frac{1}{2}w_1 ah_{i}, \quad i=1, \cdots, n.
\end{align*}
Since
\begin{align*}
G_{ij}  =& \frac{{2w_k w_{kij}  + 2w_{ki} w_{kj} }}{{\left| {Dw} \right|^2 }} - \frac{{4w_k w_{ki} w_l w_{lj} }}{{| Dw |^4 }}   + ah_{ij}  \\
=& \frac{{2w_{1ij} }}{{w_1 }} + \frac{{2w_{ki} w_{kj} }}{{w_1 ^2 }}  - a^2h_i h_j  + ah_{ij}.
\end{align*}

Then
\begin{align*}
0\ge \widetilde{F}^{ij} G^{ij}  =& \frac{{2\widetilde{F}^{ij} w_{1ij} }}{{w_1 }} + \frac{{2\widetilde{F}^{ij} w_{ki} w_{kj} }}{{w_1 ^2 }}  - a^2\widetilde{F}^{ij} h_i h_j  + a\widetilde{F}^{ij} h_{ij}  \\
\ge& \frac{{2\widetilde{F}^{ij} w_{1ij} }}{{w_1 }} +\frac{{2\widetilde{F}^{ij} w_{1i} w_{1j} }}{{w_1 ^2 }}- a^2\widetilde{F}^{ij} h_i h_j  + a \widetilde{F}^{ij} h_{ij}  \\
=& \frac{{2\widetilde{F}^{ij} u_{ij1}  - 2\widetilde{F}^{ij} \left( {\varphi h} \right)_{ij1} }}{{w_1 }} - \frac{1}{2}a^2  \widetilde{F}^{ij} h_i h_j  + a\widetilde{F}^{ij} h_{ij}  \\
\ge& \frac{2(\varepsilon fu_1+f_1)e^{s_\varepsilon+\varepsilon u}}{{w_1 }} - \frac{ C_{26}}{{w_1 }}\sum\limits_{i = 1}^n {\widetilde{F}^{ii} }  - \frac{1}{2}a^2 \widetilde{F}^{ij} h_i h_j  + a\widetilde{F}^{ij} h_{ij}  \\
\ge&  - \varepsilon C_{27} - C_{28} \frac{1}{{w_1 }}\sum\limits_{i = 1}^n {\widetilde{F}^{ii} }  + a( {\kappa_2  - \frac{a\kappa_1}{2}})\sum\limits_{i = 1}^n {\widetilde{F}^{ii} }  \\
\ge& - \varepsilon C_{27} + (\frac{a\kappa_2}{2}- C_{28} \frac{1}{{w_1 }}) \sum\limits_{i = 1}^n {\widetilde{F}^{ii} },
\end{align*}
hence we can get $w_1(x_0)$ is bounded if we choose $\varepsilon$ sufficiently small. Then we can get
\begin{align*}
|Du_{\varepsilon, s_\varepsilon}| \leq |D w|+ |\varphi||h| \leq C_{29}.
\end{align*}

\textbf{Step 3:} From the choice of $s_\varepsilon$, we know $u_{\varepsilon, s_\varepsilon}(y_0)=u_0(y_0)$. Then we have that
\begin{align*}
|u_{\varepsilon, s_\varepsilon}|_{C^0(\overline\Omega)} = u_{\varepsilon, s_\varepsilon} (x_1) \leq u_{\varepsilon, s_\varepsilon}(y_0) + |Du_{\varepsilon, s_\varepsilon}|_{C^0} |x_1 - y_0| =u_0(y_0)+ |Du_{\varepsilon, s_\varepsilon}|_{C^0} |x_1 - y_0| \le C_{30}.
\end{align*}
And the second order estimate now holds by the same calculations in \cite{CZ16}. Thus we have the higher order estimates as in \cite{LT86}. Therefore by extracting subsequence, we have $s_{\varepsilon_i}$ converges to $s_\infty$ and $u_{\varepsilon, s_\varepsilon}$ converges to a $k$-convex function $u_{ell}^\infty$ which satisfies equation\eqref{8.1} with $s = s_\infty$.

\end{proof}

\subsection{A priori estimates of \eqref{1.9}}

In this subsection, we prove the following a priori estimates of \eqref{1.9}.

\textbf{(1) $u_t$-estimate.}

Following the proof of \eqref{3.1} in Lemma \ref{lem3.1}, we can get
\begin{align} \label{8.5}
|u_t(x,t) | \leq \max|u_t(x,0) |  \leq C_{31}, \quad \forall (x,t) \in \Omega \times [0, T),
\end{align}
where $C_{31}$ depends only on $n$, $k$, $l$, $\min f$, $|f|_{C^0}$ and $|u_0|_{C^2}$.

\textbf{(2) $|Du|$ estimate.}

For any $T_0 \in (0, T)$, we will prove that
\begin{align}\label{8.6}
\max\limits_{\overline{\Omega }\times [0, T_0]} |Du| \leq  C_{32},
\end{align}
where $C_{32}$ depends only on $n$, $k$, $l$, $\Omega$, $|Du_0|_{C^0}$, $\min f$, $|f|_{C^1}$ and $|\varphi|_{C^3}$, but is independent of $|u|_{C^0}$ and $T_0$.

Since $\Omega$ is smooth and strictly convex, there exist a defining function $h \in C^\infty (\overline{\Omega})$ and positive constants $a_0$ and $A_0$ such that
\begin{align*}
& h = 0 \quad \text{ on } \partial \Omega, \quad h < 0 \quad \text{ in } \Omega; \\
& |D h| = 1 \quad \text{ on } \partial \Omega,\quad  |D h| \leq 1 \quad \text{ in } \Omega; \\
& 0 < a_0 I_n \leq D^2 h \leq A_0 I_n  \quad \text{ in } \Omega.
\end{align*}

Denote $w (x,t) = u(x,t)- \varphi(x) h(x)$, and we consider the following auxiliary function in $\overline{\Omega }\times [0, T_0]$
\begin{align*}
 G (x,t) = \log \left| {Dw} \right|^2  + a h,
\end{align*}
where $a=\frac{a_0}{2}$. Suppose that $G$ attains its maximum at the point $(x_0, t_0) \in \overline{\Omega }\times [0, T_0]$. If $t_0 =0$, then
the a priori estimate holds directly. In the following, we always assume $t_0 >0$. Firstly, we claim that $x_0\in \Omega$. In fact, $x_0\in\partial\Omega$, we assume $x_0 =0$ and choose the coordinate such that $\partial\Omega\cap B_{\delta}(x_0)$ can be represented as $(x', x_n)$ with $x_n=\rho(x')$, where $\rho(x')$ satisfies $\rho(x_0')=0$ and $D_{x'}\rho(x_0')=0$. Also we have $\nu(x_0)=(0,\cdots,0,-1) = Dh(x_0)$, and then $w_n(x_0, t_0) = u_n (x_0, t_0)-\varphi \cdot (-1) = -u_\nu +\varphi =0$. Rotating the $x'$-axis, we can further assume that $w_1(x_0, t_0)=|Dw|(x_0, t_0)$. Moreover we have $u_1(x_0, t_0)=w_1(x_0, t_0)$ and $u_i(x_0, t_0)=w_i(x_0, t_0)=0$ for $2\le i\le n-1$. Then we can get
\begin{align*}
0\geq  -\frac{\partial G}{\partial \nu}(x_0, t_0)=&\frac{\partial G}{\partial x_n}(x_0, t_0)\\
=&\frac{2w_kw_{kn}}{|Dw|^2}+ah_n
=\frac{2w_{1n}}{w_1} - a\\
=&\frac{2u_{1n}}{w_1} - \frac{2\varphi_{1}h_n+2\varphi_{n}h_1+2\varphi h_{1n}}{w_1} - a\\
=&\frac{2u_{1n}}{w_1}-\frac{-2\varphi_{1}+2\varphi h_{1n}}{w_1} - a\\
=&2{\nu^1}_1 - a \geq 2 a_0 - a > 0,
\end{align*}
where we have used the equality $u_{n1}=u_{k}{\nu^k}_1-\varphi_{1}=u_{1}{\nu^1}_1+\varphi{\nu^n}_1-\varphi_{1}$. Contradiciton.

Hence $x_0\in\Omega$, and we can choose the coordinate such that $w_1(x_0, t_0)=|Dw|(x_0, t_0)$. Then we have
\begin{align*}
0 = G_i  = \frac{{2w_k w_{ki} }}{{\left| {Dw} \right|^2 }} + ah_i  = \frac{{2w_{1i} }}{{w_1 }} + ah_i, \quad i=1, \cdots, n,
\end{align*}
and
\begin{align*}
0 \leq G_t  = \frac{{2w_k w_{kt} }}{{\left| {Dw} \right|^2 }}  = \frac{{2u_{1t} }}{{w_1 }}.
\end{align*}
Hence
\begin{align*}
\frac{{w_{1i} }}{{w_1 }} =  - \frac{1}{2}ah_{i}, \quad  w_{1i}  =  - \frac{1}{2}w_1 ah_{i}, \quad i=1, \cdots, n.
\end{align*}
Since
\begin{align*}
G_{ij}  =& \frac{{2w_k w_{kij}  + 2w_{ki} w_{kj} }}{{\left| {Dw} \right|^2 }} - \frac{{4w_k w_{ki} w_l w_{lj} }}{{| Dw |^4 }}   + ah_{ij}  \\
=& \frac{{2w_{1ij} }}{{w_1 }} + \frac{{2w_{ki} w_{kj} }}{{w_1 ^2 }}  - a^2h_i h_j  + ah_{ij}.
\end{align*}

Denote $F^{ij}  =: \frac{\partial \big(\log \frac{\sigma _k(D^2u)}{\sigma _l(D^2u)}  \big)}{\partial u_{ij}}$. By the equation and $u_t$ estimate, we have
\begin{align} \label{8.7}
\sum\limits_{i = 1}^n {F^{ii} }  =&\frac{\sigma _k}{\sigma _l} \frac{{(n - k + 1)\sigma _{k - 1}\sigma _l - (n - l + 1)\sigma _k \sigma _{l - 1} }} {{\sigma _l ^2 }}  \notag \\
\ge& c(n,k)\big(\frac{\sigma _k(D^2u)}{\sigma _l(D^2u)}  \big)^{ \frac{-1}{k-l}}  = c( {n,k} )f^{  \frac{-1}{k-l}} e^{ \frac{-1}{k-l} u_t}  \ge c_9 >0.
\end{align}
Then
\begin{align*}
0\ge F^{ij} G^{ij} -G_t  =& \frac{{2F^{ij} w_{1ij} }}{{w_1 }} + \frac{{2F^{ij} w_{ki} w_{kj} }}{{w_1 ^2 }}  - a^2F^{ij} h_i h_j  + aF^{ij} h_{ij} -\frac{{2u_{1t} }}{{w_1 }} \\
\ge& \frac{{2F^{ij} w_{1ij} }}{{w_1 }} +\frac{{2F^{ij} w_{1i} w_{1j} }}{{w_1 ^2 }}- a^2F^{ij} h_i h_j  + a F^{ij} h_{ij}  -\frac{{2u_{1t} }}{{w_1 }}\\
=& \frac{{2F^{ij} u_{ij1}  - 2F^{ij} \left( {\varphi h} \right)_{ij1} }}{{w_1 }} - \frac{1}{2}a^2  F^{ij} h_i h_j  + aF^{ij} h_{ij}  -\frac{{2u_{1t} }}{{w_1 }}\\
\ge& \frac{2f_1 / f}{{w_1 }} - \frac{ C_{33}}{{w_1 }}\sum\limits_{i = 1}^n {F^{ii} }  - \frac{1}{2}a^2 F^{ij} h_i h_j  + aF^{ij} h_{ij}  \\
\ge&  -  \frac{2|f_1 |/ f}{{w_1 }}- C_{34} \frac{1}{{w_1 }}\sum\limits_{i = 1}^n {F^{ii} }  + a( {a_0  - \frac{a}{2}})\sum\limits_{i = 1}^n {F^{ii} } ,
\end{align*}
hence we can get $w_1(x_0,t_0)$ is bounded, and then \eqref{8.6} holds.

\textbf{(3) $|D^2 u|$ estimate.}

Following the proof of Theorem \ref{th6.1}, we can get
\begin{align}
|D^2 u|   \leq C_{35}, \quad \forall (x,t) \in \Omega \times [0, T),
\end{align}
where $C_{35}$ depends only on $n$, $k$, $l$, $\Omega$, $|u_0|_{C^2}$, $|D u|_{C^0}$, $\min f$, $|f|_{C^2}$ and $|\varphi|_{C^3}$, but is independent of $|u|_{C^0}$.

\subsection{Proof of Theorem \ref{th1.2}}:

\textbf{(1) We first get a bound for the solution $u$ of \eqref{1.9}.}

Now denote $u^\infty( x,t) = u_{ell}^\infty  \left( x \right) +s^\infty t$, where $u_{ell}^\infty$ is the solution obtained  in Theorem~\ref{th8.1}, then $u^\infty$ satisfies
\begin{align}
 \left\{ \begin{array}{l}
   u^\infty  _t  = \log \frac{\sigma _k ( {D^2 u^\infty  })}{\sigma _l ( {D^2 u^\infty  })} - \log f,  \quad (x,t) \in \Omega \times (-\infty, + \infty),  \\
   \frac{{\partial u^\infty  }}{{\partial \nu }} = \varphi \left( x \right), \qquad \qquad \quad \qquad (x,t) \in \partial \Omega \times (-\infty, + \infty),  \\
   u^\infty  \left( {x,0} \right) = u_{ell}^\infty (x),  \qquad\quad  \qquad x \in \Omega.  \\
\end{array} \right.
\end{align}

Take $C_{36} = \mathop {\max }\limits_{\bar \Omega } |u_{ell}^\infty  | + \mathop {\max }\limits_{\bar \Omega } \left| {u_0 } \right|$, then \begin{align}
u_{ell}^\infty  \left( x \right) - C_{36}  \le u \left( {x,0} \right) = u_0 ( x ) \le u^\infty  ( x,0) + C_{36},  \quad \forall x\in \Omega.\notag
\end{align}
Thus by parabolic maximum principle we obtain
\begin{align}
u_{ell}^\infty  - C_{36}  \le u(x, t) \le u_{ell}^\infty  + C_{36}, \quad  (x, t) \in \Omega \times  (0, + \infty). \notag
\end{align}
That is, we obtain the $C^0$ estimate of $u$
\begin{align}\label{8.10}
s^\infty t - C_{37}  \le u(x, t) \le s^\infty t + C_{37}, \quad (x, t) \in \Omega \times  (0, + \infty),
\end{align}
where $C_{37} =C_{36} +\mathop {\max }\limits_{\bar \Omega } |u_{ell}^\infty  |$.

\textbf{(2) We prove the solution $u$ of \eqref{1.9} is longtime existence and smooth.}

The $C^1$ and $C^2$ estimates hold as in Subsection 8.2, and the $C^0$ estimate is established as above.
Following the discussions in \cite{SS03}, we can obtain the existence of the smooth $k$-admissible solution $u(x,t)$, and all higher derivatives of $u$ have uniform bounds. By the uniform estimates of $u$ and the uniform parabolicity of equation \eqref{1.9}, the solution $u(x,t)$ exists for all time, that is $T = +\infty$.

\textbf{(3) Now we will show that $u$ converges to a translating solution as $t \rightarrow + \infty$.}

To obtain the convergence, we just need to prove that there exists a constant $a$ such that
\begin{align} \label{8.11}
\lim\limits_{t \rightarrow +\infty} |u ( {x,t} )- u_{ell}^\infty  \left( x \right) - s^\infty t  -a |_{C^m(\Omega)} =0,
\end{align}
holds for any integer $m \geq 0$.

Obviously, $(u_{ell}^\infty,  s^\infty)$ is a solution of \eqref{1.9}, then $(u_{ell}^\infty +a,  s^\infty)$ is also a solution of \eqref{1.9}.

We denote $w(x,t):= u(x,t) - u^\infty(x,t)$, then it satisfies
\begin{align}
\left\{ \begin{array}{l}
    w_t  = a^{ij} w_{ij} ,\quad (x,t) \in \Omega \times (0, + \infty),  \\
   \frac{{\partial w}}{{\partial \nu }} = 0, \qquad \quad  x \in \partial \Omega.  \\
\end{array} \right.
\end{align}
where $a^{ij}$ is positive definite. If there exits some time $t_0$ such that $w$ is constant in $\Omega\times[t_0, \infty)$, i.e. $u=u^{\infty}$ in $\Omega\times[t_0, \infty)$. Thus $u$ is a translating solution. If for any $t>0$, $w$ is not  constant in $\Omega\times[t, \infty)$. We claim that $\mathbf{osc~ } w(\cdot,t) = \max \limits_{\overline\Omega}w(x,t)-\min \limits_{\overline\Omega}w(x,t)$ is strictly decreasing. In fact, for any $t_1<t_2$, there hold by the maximum principle and Hopf Lemma,
\begin{align}
\max\limits_{\bar \Omega } u( {x,t_1 }) > \max \limits_{\bar \Omega } u( {x,t_2 }),
\end{align}
and
\begin{align}
\min\limits_{\bar \Omega } u( {x,t_1 }) < \min \limits_{\bar \Omega } u( {x,t_2 } ).
\end{align}
Thus $\mathbf{osc~ }{w(\cdot,t_1)} > \mathbf{osc~ } {w(\cdot,t_2)}$, i.e. $\mathbf{osc~ }{w(\cdot,t)}$ is strictly decreasing. Hence we can get
\begin{align}
\lim\limits_{t \to \infty } \mathbf{osc~ }{w(\cdot,t)} = \delta \geq 0. \notag
\end{align}

 In the following, we prove $\delta=0$. We define $u^{i}(x,t):=u(x,t+t_i)-s^\infty t_i$ for a sequence $\{t_i\}$ which converges to $\infty$. Since \eqref{8.10}, we have $-C_{37} + ts^\infty   \le u^{i}(x,t)\le C_{37} + ts^\infty$. And $|u^i|_{C^k}\le C$, for any $k\ge 1$. Hence, there exists a subsequence (for convenience we also denote) $u^i$ such that $u^i$ converges locally uniformly in any $C^k$-norm to a $k$-convex function $u^*$. Moreover, $u^*$ exists for all time $t \in (-\infty, + \infty)$ and satisfies the following equation
\begin{align}
 \left\{ \begin{array}{l}
   u^* _t  = \log \frac{\sigma _k ( {D^2 u^* })}{\sigma _l ( {D^2 u^* })} - \log f,  \quad (x,t) \in \Omega \times (-\infty, + \infty),  \\
   \frac{{\partial u^* }}{{\partial \nu }} = \varphi ( x ), \qquad \qquad \quad \qquad (x,t) \in \partial \Omega \times (-\infty, + \infty).  \\
\end{array} \right.
\end{align}
So for any time $t \in (-\infty, + \infty)$, we have
\begin{align}
\mathbf{osc~ }{(u^*-u^\infty)(\cdot,t)}=&\lim\limits_{i \to \infty }\mathbf{osc~ }((u(\cdot,t+t_i)-s^\infty t_i-u^\infty(\cdot,t)) \notag \\
=&\lim \limits_{i \to \infty }\mathbf{osc~ }((u(\cdot,t+t_i)-u^\infty(\cdot,t+t_i)) \notag \\
=&\lim \limits_{i \to \infty }\mathbf{osc~ }(w(\cdot,t+t_i)) \notag \\
=&\delta. \notag
\end{align}
Namely,
\begin{align} \label{8.16}
\max \limits_{\overline\Omega}(u^*-u^\infty)- \min \limits_{\overline\Omega}(u^*-u^\infty) \equiv \delta,
\end{align}
holds for any time $t \in (-\infty, + \infty)$. It is easy to know $u^*-u^\infty$ satisfies
 \[
\left\{ \begin{array}{l}
 \left( {u^* - u^\infty  } \right)_t  = a^{ij} \left( {u^* - u^\infty  } \right)_{ij},\quad (x,t) \in \Omega \times (-\infty, + \infty),  \\
 \frac{{\partial \left( {u^* - u^\infty  } \right)}}{{\partial \nu}} = 0,\qquad \qquad \qquad \qquad (x,t) \in \partial \Omega \times (-\infty, + \infty). \\
 \end{array} \right.
\]
Then $\max \limits_{\overline\Omega}(u^*-u^\infty)$ is decreasing with respect to $t$, and $\min \limits_{\overline\Omega}(u^*-u^\infty)$ is increasing by the maximum principle. Hence from \eqref{8.16}, $\max \limits_{\overline\Omega}(u^*-u^\infty)$ and $\min \limits_{\overline\Omega} (u^*-u^\infty)$ are constants for any time $t \in (-\infty, + \infty)$. By the strong maximum principle and Hopf Lemma, $u^*-u^\infty$ is a constant, and this implies $\delta=0$.

Now, we have $\lim\limits_{t \to \infty }\mathbf{osc~ }{w(\cdot,t)}=0$, then there exists a constant $a$ such that $\lim\limits_{t \to \infty }\max \limits_{\overline\Omega}(u-u^\infty)=\lim\limits_{t \to \infty }\min \limits_{\overline\Omega}(u-u^\infty)= a$. Thus $\lim\limits_{t \to \infty }|u(\cdot,t)-u^\infty(\cdot,t) -a |_{C^0(\overline\Omega)}=0$. The $C^1$-norm convergence follows by the following interpolation inequality
\begin{align*}
|D v|^2_{C^0(\overline\Omega)}\le c(\Omega)|v|_{C^0(\overline\Omega)}(|D^2 v|_{C^0(\overline\Omega)}+|Dv|_{C^0(\overline\Omega)})
\end{align*}
for $v=u-u^\infty -a$. The $C^k$-norm convergence is similar. Hence \eqref{8.11} holds, which means the solution $u(x,t)$ converges to a translating solution as $t \rightarrow + \infty$. The proof of Theorem \ref{th1.2} is finished.

\end{document}